\documentclass[11pt,twoside,a4paper]{amsart}
\usepackage{amssymb}

\usepackage{amsmath}
\usepackage{amsthm,amssymb,latexsym}
\usepackage[all,cmtip]{xy}
\usepackage{url,ae}
\usepackage{t1enc}
\usepackage{fancyheadings}
\usepackage{mathrsfs}
\usepackage{graphicx}
\usepackage{eufrak} 
\usepackage{geometry}

\date{\today}

\def\1{{\bf 1}}
\def\div{{\rm div}}
\def\ett{{\bf 1}}

\def\GZ{{\mathcal{GZ}}}
\def\Zy{{\mathcal Z}}

\def\deg{\text{deg}\,}

\def\w{\wedge}

\def\dbar{\bar\partial}

\def\C{{\mathbb C}}
\def\w{{\wedge}}
\def\P{{\mathbb P}}

\def\A{{\mathcal A}}
\def\B{{\mathcal B}}

\def\codim{{\rm codim\,}}

\def\Z{{\mathbb Z}}

\def\L{{\mathcal L}}

\def\Re{{\rm Re\,  }}

\def\L{{\mathcal L}}
\def\U{{\mathcal U}}

\def\J{{\mathcal J}}
\def\nbh{neighborhood }

\def\be{\begin{equation}}
\def\ee{\end{equation}}

\def\Ok{\mathcal O}
\def\mult{{\rm mult}}
\def\b{{\bullet}}
\def\Pr{{\mathbb P}}
\def\p{{\mathfrak p}}
\def\Pk{{\mathbb P}}
\def\s{{\circ}}
\def\fff{{\sigma}}

\newtheorem{thm}{Theorem}[section]
\newtheorem{lma}[thm]{Lemma}

\newtheorem{prop}[thm]{Proposition}

\theoremstyle{definition}

\newtheorem{df}[thm]{Definition}

\theoremstyle{remark}

\newtheorem{preremark}[thm]{Remark}
\newtheorem{preex}[thm]{Example}

\newenvironment{remark}{\begin{preremark}}{\qed\end{preremark}}
\newenvironment{ex}{\begin{preex}}{\qed\end{preex}}

\numberwithin{equation}{section}

\usepackage{xcolor}

\title[Nonproper intersection products and generalized cycles]{Nonproper intersection products and generalized cycles}

\begin{document}

\date{\today}

\author[Andersson \& Eriksson \& Samuelsson Kalm \&
Wulcan \&  Yger
]{Mats Andersson \& Dennis Eriksson \& H\aa kan Samuelsson Kalm \&
Elizabeth Wulcan \&  Alain Yger}

\address{Department of Mathematical Sciences\\Chalmers University of Technology and University of
Gothenburg\\SE-412 96 G\"OTEBORG\\SWEDEN}

\email{matsa@chalmers.se, dener@chalmers.se, hasam@chalmers.se, wulcan@chalmers.se}

\address{Institut de Math\'ematique \\
Université Bordeaux 1 \\
33405, Talence \\
France}

\email{Alain.Yger@math.u-bordeaux.fr}


\thanks{The first, third and fourth  author  were
  partially supported by the Swedish
  Research Council}

\begin{abstract}
In this article we develop intersection theory in terms of the $\B$-group 
of a reduced analytic space. This group was introduced in a previous work as an analogue of the Chow group; it is 
generated by currents that are direct images 
of Chern forms and it contains all usual cycles. However,
contrary to Chow classes, the $\B$-classes have well-defined multiplicities at each point. 

We focus on a $\B$-analogue of the intersection theory based on the 
St\"uckrad-Vogel procedure and the join construction in projective space. Our 
approach provides global $\B$-classes which satisfy a B{\'e}zout theorem and have the expected local intersection 
numbers.
An essential feature 
is that we take averages, over various auxiliary choices, by integration.
We also introduce $\B$-analogues of more classical constructions of intersections 
using the Gysin map of the diagonal. These constructions are connected via a $\B$-variant of van~Gastel's formulas.
Furthermore, we prove that our intersections coincide with the classical ones on cohomology level.

\end{abstract}


\maketitle

\section{Introduction}

Let $Y$ be a smooth manifold of dimension $n$.
Assume that $\mu_1,\ldots,\mu_r$ are cycles on $Y$ of
pure codimensions $\kappa_1,\ldots,\kappa_r$, respectively, let $\kappa:=\kappa_1+\cdots +\kappa_r$,
and let  $V$ be the  set-theoretic intersection $V=|\mu_1|\cap\cdots\cap|\mu_r|$.
If $\mu_j$ intersect properly, that is, if $\codim V=\kappa$, 
then there is a well-defined intersection cycle
\begin{equation*}
\mu_1\cdot_Y  \cdots \, \cdot_Y\mu_r=\sum m_j V_j,
\end{equation*}
where $V_j$ are the irreducible components of $V$ and $m_j$ are integers.
In the nonproper case there is no canonical intersection cycle. However, following Fulton-MacPherson, see \cite{Fult},
there is an intersection product  $\mu_1\cdot_Y\cdots \, \cdot_Y\mu_r$, which is an
element in the Chow group $\A_{n-\kappa}(V)$;   that is, the product is represented by a cycle on
$V$ of dimension $n-\kappa$
that is determined up to rational equivalence.  For instance, the self-intersection of a line $L$ in $\P^n$ is
obtained by intersecting $L$ with  a perturbation of $L$. If $n=2$ one gets 
an arbitrary point on $L$, whereas if $n\geq 3$ the intersection is empty.

In case $Y=\Pk^n$ there is an intersection product due to St\"uckrad and Vogel, \cite{SV,V},
that in general consists of components of various dimensions. 
For instance the self-intersection of a line is actually the line itself independently of $n$. However,
in general a nonproper intersection has so-called moving components,
that are only determined up to rational equivalence.   There is a
relation to the classical (Fulton-MacPherson) intersection product via
van~Gastel's formulas, \cite{gast},  see also \cite{FOV}.

Tworzewski, \cite{Twor}, introduced, for $x\in V$,
{\it local intersection numbers}
 \begin{equation}\label{locsec}
\epsilon_\ell(\mu_1,\ldots,\mu_r, x), \quad \ell=0,\ldots, \dim V, 
\end{equation}
see also \cite{GG, AM, AR} and Section~\ref{segretal} below. 
In the proper case $\epsilon_\ell(\mu_1,\ldots,\mu_r, x)$ is precisely the 
multiplicity at $x$ of the proper intersection $\mu_1\cdot_Y \cdots \cdot_Y \mu_r$ for
$\ell=\dim V$ and $0$ otherwise. In the nonproper case the intersection numbers may be nonzero 
also for $\ell<\dim V$.  In general no representative of the classical intersection product,
cf.~\cite[Remark~1.4]{aswy}, or representative of the St\"uckrad-Vogel product, 
can represent these numbers at all points.

The main objective of this paper is to introduce 
a product of cycles in $\P^n$ that at each point carries the local intersection
numbers and at the same time have reasonable global properties, such as 
respecting the B{\'e}zout formula.  To this end we must extend the class of cycles,
and our construction is based on the $\Z$-module $\GZ(X)$ 
of  {\it generalized cycles}  on a (reduced) analytic space $X$ introduced in \cite{aeswy1}.
It is the smallest class of currents on  analytic spaces that is closed under
multiplication by components of Chern forms and under direct images under proper holomorphic
mappings. It turns out that generalized cycles inherit a lot of geometric properties and
preferably can be thought of as geometric objects. Actually we are primarily interested in
a  certain natural quotient group $\B(X)$ of $\GZ(X)$.  Each $\mu$ in $\GZ(X)$ has  a well-defined
Zariski support $|\mu|\subset X$ that only depends on its class in $\B(X)$.
For a subvariety $V\hookrightarrow X$ there is a natural identification
of $\B(V)$ with elements in $\B(X)$ that have Zariski support on $V$.
The group of cycles $\Zy(X)$ is naturally embedded in  $\B(X)$.
Given $\mu\in\B(X)$ also its restriction $\1_V\mu$ to  the subvariety $V$ is an element in $\B(X)$.
Each element
in $\GZ(X)$, and in $\B(X)$, has a unique decomposition into sums of irreducible components. 
Each irreducible element  has in turn a unique decomposition
into components of various dimensions. We let $\B_k(X)$ denote the elements in $\B(X)$ of pure dimension $k$.
We also introduce a notion of {\it effective} generalized cycle $\mu$ in $\GZ(X)$, and class in $\B(X)$,
generalizing the notion of effective cycle.  
Each  $\mu$ in $\GZ(X)$, and in $\B(X)$, 
has a well-defined multiplicity, $\text{mult}_x\,\mu$, at each point $x\in X$, that is an integer and nonnegative if $\mu$ is effective.
Moreover, for each $\mu$ in $\GZ_k(X)$, or in $\B_k(X)$, there is a unique
decomposition
\begin{equation}\label{deco}
\mu=\mu_{fix}+\mu_{mov},
\end{equation}
where $\mu_{fix}$ is an ordinary cycle of dimension $k$, whose irreducible components are called
the {\it fixed} components of $\mu$,
and $\mu_{mov}$, whose irreducible components, the {\it moving} components,
have Zariski support on varieties of dimension strictly larger than $k$.

\smallskip
Each $\mu$ in $\GZ_k(\Pk^n)$, or in $\B_k(\Pk^n)$, has the {\it degree}
\begin{equation}\label{degdef}
\deg\mu:=\int_{\Pk^n}\omega^k\w \mu,
\end{equation}
where $\omega$ is the first Chern class of $\Ok(1)\to\Pk^N$, for instance represented by the Fubini-Study
metric form. If $\mu=\mu_0+\mu_1+\cdots $, where $\mu_k$ has pure
dimension $k$, then 
$$
\deg \mu:=\deg\mu_0+\deg\mu_1+\cdots.
$$ 
For each point $x\in \Pk^n$ and $\mu_1,\ldots,\mu_r\in\B(\U)$  for some open subset $\U\subset\Pk^n$ 
there are $\Z$-valued functions $\epsilon_\ell(\mu_1,\ldots,\mu_r,x)$, $\ell=0,1,\ldots$,  that are
$\Z$-multilinear in $\mu_j$, only depend on the germs of $\mu_j$ at $x$, and
which coincide with  the local intersection numbers \eqref{locsec} if $\mu_j$ are cycles. We say that 
$\epsilon_\ell(\mu_1,\ldots,\mu_r,x)$ are the {\it local intersection numbers} of $\mu_1,\ldots,\mu_r$ at $x$. 
If $\mu_j$ are effective, then these numbers are nonnegative.

\smallskip
Our main result concerns a $\Z$-multilinear mapping 
\begin{equation}\label{poker}
\B(\Pr^n)\times \cdots \times \B(\Pr^n)\to \B(\Pr^n), \quad  \mu_1,\ldots, \mu_r\mapsto \mu_1\bullet\cdots\b\mu_r.
\end{equation}
We  say that the image is the $\bullet$-product of $\mu_1,\ldots,\mu_r$. 
It is obtained, roughly speaking, in the following way:
We first choose representatives for the $\B$-classes $\mu_1, \ldots, \mu_r$,  then form a 
St\"uckrad-Vogel-type product of them. Even for cycles, this product depends on several choices.
Taking a suitable mean value, we get a generalized cycle that turns out to define
an element in $\B(\P^n)$ that is independent of all choices.  If $\mu_j$ are cycles, then
the fixed components in the St\"uckrad-Vogel product appear as fixed components of
$\mu_1\bullet\cdots \bullet\mu_r$.
The formal definition,  Definition~\ref{bulletdef},  is expressed in terms of a certain Monge-Amp\`ere type product,
that can be obtained as a limit of quite explicit expressions, see  Section~\ref{bulletsection}.  
Here is our main result.

\begin{thm}\label{mainthm}
The $\Z$-multilinear mapping \eqref{poker} has the following properties. To begin with, 
$|\mu_1\bullet\cdots\b\mu_r|\subset\cap_{j=1}^r|\mu_j|$, $\mu_1\bullet\cdots\b\mu_r$ is commutative, and 
\begin{equation}\label{stenstod}
\mult_x(\mu_1\bullet\cdots\b\mu_r)_\ell=\epsilon_\ell(\mu_1,\ldots,\mu_r, x), \quad x\in\Pk^n,\  \ell=0,1,\ldots,
\end{equation}
where $(\ \  )_\ell$ denotes the component of dimension $\ell$.
If $\mu_j$ have pure dimensions and 
\begin{equation}\label{vill}
\rho:=\sum_1^r\dim \mu_j -(r-1)n\ge 0, 
\end{equation}
then
\begin{equation}\label{likhet}
\deg(  \mu_1\bullet\cdots\b\mu_r)=\Pi_1^r\deg\mu_j.
\end{equation}
If  $\mu_j$ are effective, then $\mu_1\bullet\cdots\b\mu_r$  is effective and 
\begin{equation}\label{olikhet}
\deg(  \mu_1\bullet\cdots\b\mu_r)\le\Pi_1^r\deg\mu_j.
\end{equation}
%
%
%
If $\mu_1, \ldots, \mu_r$ are cycles that intersect
  properly, then 
\begin{equation}\label{studsmatta}
\mu_1\b\cdots\b\mu_r = \mu_1\cdot_{\P^n}\cdots\cdot_{\P^n} \mu_r. 
\end{equation}
\end{thm}

\noindent  One should keep in mind that the $\b$-product of $r$ factors is not a repeated $\b$-product of two factors.  
In general, the $\b$-product of two factors is not associative, see Example~\ref{kokong}.

\smallskip

\noindent Notice that $\rho$ equals $n-(n-\dim\mu_1 +\cdots
+n-\dim\mu_r)$, which is 
the ``expected dimension'' of the intersection. 
The B{\'e}zout formula \eqref{likhet} may hold even if  $\rho<0$:  
For instance,  if
$\mu_j$ are different lines through the point $a$, then their $\b$-product is $a$ so that both sides
of \eqref{olikhet} are $1$, see Example \ref{kaffekopp}. 
Moreover, if we take a linear
embedding $\Pk^n\hookrightarrow \Pk^{n'}$, $n'>n$, and consider $\mu_j$ as
elements in $\B(\Pk^{n'})$, then the product is unchanged. In particular, the $\b$-self-intersection
of a $k$-plane is always the $k$-plane itself.  

The $\b$-self-intersection of the cuspidal curve
$Z=\{x_1^3-x_0x_2^2=0\}$ in  $\P^2$ is in the classical sense 
represented by 
$9$ points on $Z$ obtained as the divisor of
a generic meromorphic section of $\Ok_{\P^2}(3)$ restricted to $Z$.
The fixed part of the self-intersection in the St\"uckrad-Vogel sense is the curve itself plus
$3$ times the point $a=[1,0,0]$, whereas the moving part consists of
another three points on $Z$ that are determined up to rational equivalence on $Z$. 
Our product $Z\bullet Z$ consists of the the fixed part $Z+3a$ of the
St\"uckrad-Vogel(SV)-product and a moving component $\mu$ of dimension zero
and degree $3$; we think of $\mu$ as three points ``moving
around''  on $Z$, cf.\ Example~\ref{cusp}. 
In this case the local intersection numbers are carried by the fixed components.
In general also moving components can contribute, see, e.g., Example~\ref{theta}.
\smallskip
We also consider another intersection product that is a $\B$-variant of the classical nonproper
intersection product in \cite{Fult}:
For any regular embedding $i$, in \cite{aeswy1} we introduced a $\B$-analogue of the 
Gysin mapping $i^!$ used in \cite{Fult}, see Section~\ref{guleboj} below.
Let $i\colon\Pk^n \to  \Pk^n\times \cdots \times\Pk^n=(\Pk^n)^r$ be
the diagonal embedding in $(\Pk^n)^r$. 
In analogy with the classical intersection product in \cite{Fult} we define, for pure-dimensional $\mu_j$, 
\begin{equation*}
\mu_1\cdot_{\B(\Pk^n)} \cdots\, \cdot_{\B(\Pk^n)}\mu_r:=i^!(\mu_1\times \cdots \times\mu_r) 
\end{equation*}
in $\B(\Pk^n)$.
We have the following relation to the $\b$-product.

\begin{thm}\label{main2}
Assume that $\mu_1,\ldots, \mu_r\in\B(\Pk^n)$ have pure dimensions. 
Let $V=\cap_j|\mu_j|$ 
and let $\rho$ be as in \eqref{vill}.
Then 
\begin{equation*}
\mu_1 \cdot_{\B(\P^n)}\cdots\,  \cdot_{\B(\P^n)}\mu_r= \sum_{\ell=\max(\rho,0)}^{\dim V}
 \omega^{\ell-\rho}\w (\mu_1\bullet \cdots \b\mu_r)_\ell.
\end{equation*}
\end{thm}
In particular, $\mu_1
\cdot_{\B(\P^n)}\cdots\,  \cdot_{\B(\P^n)}\mu_r=\mu_1
\cdot_{\P^n}\cdots\,  \cdot_{\P^n}\mu_r$ if $\mu_1, \ldots, \mu_r$
are cycles that intersect properly, see \eqref{studsmatta}. 


\smallskip
In \cite[Section~10]{aeswy1}  we introduced cohomology groups $\widehat H^{*,*}(V)$ for a
reduced subvariety $V\hookrightarrow\Pk^n$ of pure dimension $d$ that coincide with
usual de~Rham cohomology $H^{*,*}(V)$ when $V$ is smooth.  There are natural mappings
$\A_k(V)\to  \widehat H^{d-k,d-k}(V)$ and $\B_k(V)\to  \widehat H^{d-k,d-k}(V)$.

\begin{thm}\label{main3}
Assume that $Z_1,\ldots, Z_r$ are cycles in $\P^N$ and let $V=\cap_j|Z_j|$.
The images in $\widehat H^{*,*}(V)$ of the Chow class $Z_1\cdot_{\Pr^n} \cdots \,  \cdot_{\Pk^n}Z_r$  and the $\B$-class
$Z_1\cdot_{\B(\P^n)} \cdots \,   \cdot_{\B(\P^n)} Z_r$ coincide. 
\end{thm}

\smallskip
The plan of the paper is as follows. Sections~\ref{prel} through \ref{vogelsec}
contain mainly material from \cite{aeswy1} and well-known facts from \cite{Fult}, 
as well as the definition of  local intersection numbers and of 
the notion of an effective generalized cycle.
The product $\cdot_{\B(Y)}$ is introduced in Section~\ref{intprod}. 
In Section~\ref{bulletsection} we  define the  $\b$-product and prove 
Theorem~\ref{mainthm}, whereas the connection to the $\cdot_{\B(\Pr^n)}$-product
is worked out in Section~\ref{rela}.  Finally we have collected several 
examples in Section~\ref{exsektion}.

\smallskip
\noindent{\bf Ackowledgement:}  We are grateful to Jan Stevens for 
valuable discussions on the
ideas in this paper.

\section{Preliminaries}\label{prel}
Throughout this section $X$ is a reduced analytic space of pure dimension $n$. 
We will recall some basic notions from intersection theory that
can be found in \cite{Fult}, and some notions and results from
\cite{aeswy1}; however the material
in Section~\ref{effective} and Lemma \ref{skaldjur} is new. 
We formulate statements in terms of  coherent sheaves, rather than schemes.

\subsection{Currents and cycles}
We say that a current $\mu$ on $X$ of bidegree $(n-k, n-k)$ has {\it (complex) dimension}  $k$.
If $f\colon X'\to X$
is a proper mapping of analytic spaces, then $f_*$ is well-defined on currents
and preserves dimension.
If $\mu$ is a current on $X'$ and $\eta$ is a smooth form on $X$, then 
\begin{equation}\label{ingen}
\eta\wedge f_*\mu = f_*(f^*\eta\wedge \mu). 
\end{equation} 
If $\mu$ has order zero then $f_*\mu$ has order zero. If $V\hookrightarrow X$ is a
subvariety, then  
\begin{equation}\label{linalg}
\1_V f_*\mu=f_*(\1_{f^{-1} V} \mu). 
\end{equation}
If $V\hookrightarrow X$ has dimension $k$,
then its associated Lelong current (current of integration) $[V]$ has dimension $k$.
We will often identify $V$ and $[V]$. 
An analytic $k$-cycle $\mu$ on $X$ is a formal locally finite linear combination $\sum
a_j V_j$, where $a_j\in \Z$ and  $V_j\subset X$ are irreducible
analytic sets of dimension $k$. 
%
We denote the $\Z$-module of analytic $k$-cycles on $X$ by
$\Zy_k(X)$. 
The support $|\mu|$ of $\mu\in \Zy_k(X)$ coincides with the 
support of its associated Lelong current. 
Recall that $\mult_x \mu=\ell_x \mu$, 
where $\ell_x\mu$ denotes the Lelong number (of the Lelong current) of
$\mu\in\Zy_k(X)$ at
$x$, and $\mult_x\mu$ is the multiplicity of $\mu$ at
$x$, see \cite[Chapter~2.11.1]{Ch}.

If  $f:X'\to X$ is a proper mapping, then we have a mapping
\begin{equation}\label{monogam}
f_*\colon \Zy_k(X')\to\Zy_k(X),
\end{equation}
and the Lelong current of the direct image $f_*\mu$ is
the direct image of the Lelong current of $\mu$.  
If $i\colon V\hookrightarrow X$ is a subvariety,  
then $\mu\in \Zy_k(V)$ can be identified with the cycle 
$i_*\mu\in\Zy_k(X)$. 
The cycle $\mu\in \Zy_k(X)$ is \emph{rationally equivalent} to $0$  on $X$, 
$\mu\sim 0$, if
there are finitely many subvarieties
$i_j:V_j\hookrightarrow X$ of dimension $k+1$ and  
non-trivial meromorphic functions $g_j$ on $V_j$ 
such that\footnote{Here $d^c=(\partial-\dbar)/4\pi i$
so that $dd^c\log|z|^2=[0]$ in $\C$, writing $[0]$ rather than  $[\{0\}]$ for the point mass 
at $0$.}
\begin{equation*}
\mu=\sum_j (i_j)_* [\div g_j]=\sum_j (i_j)_*dd^c\log|g_j|^2=
\sum_j dd^c(\log|g_j|^2[V_j]).
\end{equation*}
We denote the \emph{Chow group} of cycles $\Zy_k(X)$ modulo rational
equivalence by $\A_k(X)$. 
 If $f:X'\to X$ is a proper morphism and $\mu\sim 0$ in $\A_k(X')$, then
$f_*\mu\sim 0$ in $\A_k(X)$ and there is an induced mapping 
 $f_*\colon \A_k(X')\to \A_k(X)$.

\subsection{Chern and Segre forms}\label{seg}
Recall that to any Hermitian line bundle\footnote{All line bundles and vector bundles
and morphism between them are assumed to be holomorphic.}
$L\to X$ there is an
associated (total) Chern form\footnote{For Chern and Segre forms (and
  classes), 
the index $k$ denotes the component of bidegree $(k,k)$, i.e., of (complex) dimension $n-k$.} 
$\hat c(L)=1+\hat c_1(L)$ and that
two  Hermitian metrics give rise to
Chern forms whose difference is $dd^c \gamma$ for a
smooth form $\gamma$ on $X$.
We let $c(L)$ denote the associated cohomology class.

Assume that $E\to X$ is a Hermitian vector bundle, and let
$\pi\colon \P(E)\to X$ be the projectivization of $E$, i.e., the
projective bundle of lines in $E$.  
Let $L=\Ok(-1)$ be the tautological
line bundle in the pullback $\pi^*E\to\Pr(E)$, and let $\hat c(L)$ be the induced
Chern form on $\Pr(E)$. Since $\pi$ is a submersion, 
$\hat s(E):=\pi_*(1/\hat c(L))$ is a smooth form on $X$ called the Segre
form of $E$.   If $E$ is a line bundle, then $\P(E)\simeq X$ and hence
\begin{equation}\label{smal}
\hat c(E)=1/\hat s(E).
\end{equation}
For a general Hermitian $E\to X$ we take \eqref{smal} as the definition
of its associated Chern form. 
If $f\colon X'\to X$
is a proper mapping, then 
\begin{equation}\label{sol}
\hat c_\ell(f^*E)=f^*\hat c_\ell(E).
\end{equation}
Since $\pi$ is a submersion two different metrics on $E$ give rise to  Segre forms and Chern forms
that differ by $dd^c\gamma$ for a smooth form $\gamma$ on $X$.  
The induced cohomology classes are denoted by $s(E)$ and $c(E)$, respectively.   
%
There are induced mappings
\begin{equation*}
\A_k(X)\to \A_{k-\ell}(X),\quad \alpha\mapsto c_\ell(E)\cap \alpha.
\end{equation*}


\subsection{Generalized cycles} \label{potta}
Generalized cycles on $X$ were introduced in  \cite{aeswy1} and all
statements in this subsection except Lemma \ref{skaldjur} 
are proved in \cite[Sections 3 and 4]{aeswy1}. 
We say that a current $\mu$ is
a {\it generalized cycle} if it is a locally finite 
linear combination over $\Z$ 
of currents of the form
$
\tau_* \alpha,
$
where $\tau\colon W\to X$ is a proper map, $W$ is smooth and connected, and
$\alpha$ is a product of  components of  Chern forms for various Hermitian
vector bundles over $W$, 
i.e.,  
\begin{equation}\label{seedan}
\alpha=\hat c_{\ell_1}(E_1)\wedge \cdots \wedge \hat c_{\ell_r}(E_r), 
\end{equation} 
where $E_j$ are Hermitian vector bundles over $W$.  One can just as well use components of Segre forms,
and one can in fact assume that all $E_j$ are line bundles.

Notice that a generalized cycle is a real closed current of order zero  with
components of bidegree $(*,*)$. 
We let $\GZ_k(X)$ denote the $\Z$-module of generalized cycles of (complex) \emph{dimension} $k$
(i.e., of bidegree $(n-k,n-k)$) and we let $\GZ(X)=\bigoplus
\GZ_k(X)$. 
If $\mu\in\GZ(X)$ and $\gamma$ is a component of a Chern form on $X$, then
$\gamma\w\mu\in\GZ(X)$. If $E\to X$ is a Hermitian vector bundle we thus have mappings
$\GZ_k(X)\to \GZ_{k-\ell}(X)$ defined by $\mu\mapsto \hat c_\ell(E)\w\mu$.
If $i\colon V\hookrightarrow X$ is a subvariety and $\mu\in\GZ(X)$,  then
$
\1_V\mu\in\GZ(X).
$
More precisely, if 
\begin{equation}\label{husmus}
\mu=\sum_j(\tau_j)_* \alpha_j,
\end{equation}
where $\tau_j:W_j\to X$, then 
$$
\1_V \mu =\sum_{\tau_j(W_j)\subset V} (\tau_j)_*\alpha_j.
$$
Each subvariety of $X$ is a generalized cycle so we have an embedding
$$
\Zy_k(X)\to \GZ_k(X).
$$
%
Given $\mu\in\GZ(X)$ there  is a smallest variety $|\mu|\subset X$,
the \emph{Zariski support} of $\mu$, such that
$\mu$ vanishes outside $|\mu|$. 
If $f\colon X'\to X$ is  proper, then we have a natural mapping
$$
f_*\colon \GZ_k(X')\to \GZ_k(X)  
$$
that coincides with \eqref{monogam}  on $\Zy_k$.
If $i\colon V \hookrightarrow  X$ is a subvariety, then
\begin{equation}\label{linser}
i_*\colon \GZ_k(V)\to \GZ_k(X)
\end{equation} 
is an injective mapping whose image is precisely those $\mu\in\GZ_k(X)$ such that
$|\mu|\subset V$.
Thus we can
identify $\GZ(V)$ with generalized cycles in $X$ that have Zariski support on $Z$.
We have the 

\smallskip
\noindent {\it Dimension principle}: 
{\it Assume that  $\mu\in\GZ_k(X)$ has Zariski support on a variety
$V$. If $\dim V =k$, then  $\mu\in\Zy_k(X)$.
If $\dim V<k$, then $\mu=0$.}

\smallskip


A nonzero generalized cycle $\mu\in\GZ(X)$ is \emph{irreducible} if
$|\mu|$ is irreducible and   
$\1_V\mu=0$ for any proper analytic subvariety $V\hookrightarrow |\mu|$.
If $\mu$ has Zariski support $V\subset X$ it is irreducible if and only if
$V$ is irreducible and $\mu$ has a representation \eqref{husmus}  where 
$\tau_j(W_j)=V$ for each $j$.
An irreducible $\mu\in \GZ(X)$  has the decomposition  
$
\mu=\mu_p+ \cdots + \mu_1 +\mu_0,   \, \mu_k\in\GZ_k(X), $ where $p$
is the dimension of $|\mu|$. 
%
%
Each $\mu\in\GZ(X)$ has a unique decomposition
\begin{equation*}
\mu =\sum_\ell\mu_\ell,
\end{equation*}
where  $\mu_\ell$ are  irreducible with different Zariski supports.

\smallskip

If 
$0\to S\to E\to Q\to 0$ is a short exact sequence of Hermitian vector bundles over $X$, then
we say that
$
\hat c(E)-\hat c(S)\w\hat c(Q)
$
is a $B$-form. 
If $\beta$ is a component of a $B$-form, then 
there is a smooth form $\gamma$ on $X$ such that
$dd^c\gamma=\beta$. 
We say that  $\mu \in \GZ_k(X)$ is equivalent to $0$ in $X$,
$\mu\sim 0$,  if $\mu$ is a locally finite sum of currents of the form
\begin{equation}\label{kusin}
\rho=\tau_* (\beta \w \alpha) = dd^c \tau_*(\gamma\wedge \alpha),
\end{equation}
where $\tau\colon W\to X$ is proper,  $\beta$ is
a component
of a $B$-form,  $\alpha$ is a product of components of Chern or Segre
forms, and $\gamma$ is a smooth form on $W$.   
If $\mu=\mu_0+\cdots +\mu_n$, where $\mu_k\in \GZ_k(X)$ we say that
$\mu\sim 0$ if $\mu_k\sim 0$ for each $k$. Let $\B(X)$ denote the
$\Z$-module of generalized cycles on $X$ modulo this equivalence. A
class $\mu\in \B(X)$ has \emph{pure dimension} $k$, $\mu\in \B_k(X)$,
if $\mu$ has a representative in $\GZ_k(X)$. Thus  $\B(X)=\oplus_k\B_k(X)$. 
%
%
The mapping $\Zy(X)\to \B(X)$ is injective
so we can consider $\Zy(X)$ as a subgroup of $\B(X)$.

If $\mu\in\B(X)$ and $\hat \mu\in\GZ(X)$ is a representative for $\mu$, then 
the \emph{Zariski support} $|\mu|\subset X$ of $\mu$ is the union of the Zariski 
supports of the irreducible components of $\hat\mu$ that are  nonzero
in $\B(X)$.  
Moreover, $\mu\in \B(X)$ is \emph{irreducible} if there is a
representative $\hat \mu\in \GZ(X)$ that is irreducible. 
The decomposition into irreducible components, as well as the
decomposition into components of different dimensions,  
extend from $\GZ(X)$ to $\B(X)$. 
%

%
If $0\to S\to E\to Q\to 0$ is a short exact sequence of Hermitian vector bundles
and $\hat\mu\in\GZ(X)$, then 
\begin{equation}\label{trump}
\hat c_\ell(E)\w\hat\mu\sim (\hat c(S)\w\hat c(Q))_\ell\w\hat\mu.
\end{equation}
In particular,  if $E$ and $E'$ are the same vector bundle with two
different Hermitian metrics, then $\hat c_\ell(E)\w\hat\mu\sim \hat c_\ell(E')\w\hat\mu$ 
so we have mappings
$$
\B_k(X)\to \B_{k-\ell}(X), \quad  \mu\mapsto c_\ell(E)\w\mu. 
$$
If $f\colon X'\to X$ is a proper mapping, then we have a natural mapping
$$
f_* \colon \B(X')\to \B(X).
$$
If $i\colon V\hookrightarrow X$ is a subvariety, then 
\begin{equation}\label{hopper}
i_*\colon \B(V)\to \B(X)
\end{equation}
 is injective,
and we can identify its image with the elements in $\B(X)$ that have Zariski support
on $V$. 
%
%

Each $\mu\in \B_k(X)$ (and $\mu\in \GZ_k(X)$)
has a unique decomposition  \eqref{deco}
where $\mu_{fix}$ is a cycle of pure dimension $k$ and
the irreducible components of $\mu_{mov}$ have Zariski supports of 
dimension strictly larger than $k$.
%
We say that the irreducible components of $\mu_{fix}$ are \emph{fixed} and that the
irreducible components of $\mu_{mov}$ are \emph{moving}. 

We will need the following simple lemma. 
\begin{lma}\label{skaldjur}
Assume that $\mu_1,\ldots, \mu_r$ are generalized cycles on reduced
analytic spaces $X_1,\ldots, X_r$. Let $p_j\colon
X_1\times \cdots \times X_r\to X_j$ be the natural projections. Then 
\begin{equation*}
\mu_1\times\cdots \times \mu_r:=p_1^*\mu_1\w\cdots\w p_r^*\mu_r
\end{equation*}
is a generalized cycle on  $X_1 \times \cdots \times X_r$. 
If $\mu_j\sim 0$ in $X_j$ for some $j$, then
$\mu_1\times\cdots\times\mu_r\sim 0$ in $X_1 \times \cdots \times X_r$. 
\end{lma}
In particular, 
for $\mu_j\in \B(X_j)$,
$j=1,\ldots, r$, there is a well-defined 
$\mu_1\times\cdots\times\mu_r\in\B(X_1\times\cdots\times X_r)$.

\begin{proof}
Assume that 
$\mu_j=(\tau_j)_* \alpha_j$, where $\tau_j\colon W_j\to X_j$ are proper and $\alpha_j$ are products
of components of Chern forms. 
Let $\pi_j\colon W_1\times \cdots \times W_r\to W_j$ be the natural
projections. 
Then $\pi_1^*\alpha_1\w\cdots\w
\pi_r^*\alpha_r$ is a product of components of Chern forms on
$W_1\times \cdots\times W_r$ and 
\begin{equation*}
\mu_1\times\cdots\times \mu_r= p_1^*\mu_1\w\cdots\w p_r^*\mu_r=(\tau_1\times\cdots\times\tau_r)_*(\pi_1^*\alpha_1\w\cdots\w \pi_r^*\alpha_r),
\end{equation*}
and hence it is a generalized cycle on $X_1 \times\cdots\times X_r$.
If, say, $\mu_1\sim 0$, we may assume,  
cf.~\eqref{kusin}, 
that $\mu_1=(\tau_1)_*(\beta\w\alpha_1)$,  where $\beta$ is a
component of a $B$-form. Then also  $\pi_1^*\beta$ is a component of a
$B$-form.
Now $\mu_1\times\cdots\times\mu_r$ is the
push-forward of $\pi_1^*\beta\w \pi_1^*\alpha_1\w
\pi_2^*\alpha_2\w\cdots\w\pi_r^*\alpha_r$ and therefore it is equivalent to $0$ in $X_1\times\cdots\times X_r$ by definition.
\end{proof}

 
\subsection{Effective generalized cycles} \label{effective}
We say that a generalized cycle $\mu$ is {\it effective} if it is a positive current, see, e.g., 
\cite[Ch.III~Definition~1.13]{Dem2}.  Clearly effectivity is preserved under direct images.  

\begin{lma} 
Let $\mu=\mu_1+\mu_2\cdots $ be the decomposition of $\mu\in\GZ(X)$ into its
irreducible components.   Then $\mu$ is effective if and only if each $\mu_j$ is effective. 
\end{lma}
 
\begin{proof}
The if-part is clear. For the converse, let $V$ be an irreducible subvariety of $X$. We already know that $\1_V\mu$ is a generalized cycle.
It is not hard to see that it is positive if $\mu$ is positive. It is also part of the Skoda-El~Mir theorem,
see, e.g., \cite[Ch.III~Theorem~2.3]{Dem2}.   Now let $V_j$ be the Zariski supports of the various $\mu_j$ and assume that
$V_k$ has minimal dimension.  Then $V_k\cap V_j$ has positive codimension in $V_j$ for each $j\neq k$.
By the definition of irreducibility it follows that $\1_{V_k}\mu=\1_{V_k}\mu_k=\mu_k$. We conclude that
$\mu_k$ is positive for each $k$ such that $V_k$ has minimal dimension. Let $V'$ be the union of these
$V_k$ and let $\mu'$ be the sum of the remaining irreducible components. Clearly 
$\mu'$ is positive in $X\setminus V'$. Let $A=ia_1 \w\bar a_1\w\ldots\w ia_r\w\bar a_r$ for smooth
$(1,0)$-forms $a_j$ and some $r$.  It follows that 
$A\w \mu'$ is positive outside $V'$ by definition. However, $\1_{V'}\mu'=0$ and so 
$A\w \mu'=A\w\1_{X\setminus V'}\mu'$ is positive. Since $A$ is arbitrary, we conclude that $\mu'$ is positive.
Now the lemma follows by induction. 
\end{proof} 
 

 We say that $\mu\in\B(X)$ is \emph{effective} if it has a representative $\hat\mu \in\GZ(X)$ that is effective.
It follows that $\mu$ is effective if and only each of its irreducible components is effective.   
Moreover, the multiplicities of an effective $\mu\in\B(X)$ are nonnegative.

\subsection{The Segre and $\B$-Segre class}\label{byggnad} 
The material in this subsection is found in \cite[Section 5]{aeswy1} or in \cite{Fult}.
Let $\J\to X$ be a coherent ideal sheaf over $X$
with zero set $Z$.
First assume that $X$ is irreducible.  If
$\J=0$ on $X$, then we define the Segre class $s(\J,X)=s_0(\J,X)=\1_X\in\mathcal{A}_n(X)$.  Otherwise,
let $\pi\colon X'\to X$ be a modification such that $\pi^*\J$ is principal\footnote{In this paper,
$\pi^*\J$ denotes the ideal sheaf on $X$ generated by the pullback of local generators of $\J$.}.
For instance 
$X'$ can be the blowup of 
$X$ along $\J$, or its normalization.
Let $D$ be the exceptional divisor, and let $L_D$ be the associated line bundle that has 
a section $\fff^0$ that defines $D$ and hence generates $\pi^*\J$. 
Then   
\begin{equation*}
s(\J,X):= \sum_{j\ge 0}(-1)^{j}
\pi_*\big(c_1(L_D)^{j}\cap [D]\big)=
\pi_*\Big(\frac{1}{1+c_1(L_D)}\cap [D]\Big);
\end{equation*}
it is a well-defined element in $\A_*(X)$. 
If $X$ has irreducible components $X_1, X_2, \ldots$, then
$s(\J,X)=s(\J,X_1)+s(\J,X_2)+\cdots$.
Notice that $s(\J,X)$ has support in $Z$ so that it can be identified with an element $s(\J,X)$
in $\A_*(Z)$.
If $\J$ is the sheaf associated with the subscheme $V$ of $X$, then $s(\J, X)$
coincides with the classical Segre class $s(V,X)$, cf.~\cite[Corollary~4.2.2]{Fult}.


We can define the $\B$-Segre class $S(\J,X)$ in an analogous way by just interpreting $\cap$
as the ordinary wedge product. However, we are interested in more explicit representations
and also in a definition of a $\B$-Segre class on $\mu\in\B(X)$. To this end 
we assume that the ideal sheaf $\J\to X$ is generated by a holomorphic section $\fff$ of a Hermitian 
vector bundle
$E\to X$. If $X$ is projective one can always find such a $\fff$ for any coherent ideal sheaf $\J\to X$.
We shall consider Monge-Amp\`ere products on a generalized
cycle $\mu$.

\begin{thm}\label{putig}
Assume that $\fff$ is a holomorphic section of $E\to X$ and let $\J$
be the associated coherent ideal sheaf with zero set $Z$.
For each $\mu\in\GZ(X)$ the limits
\begin{equation*}
(dd^c\log|\fff|^2)^k\w\mu :=\lim_{\epsilon\to 0}
\big (dd^c\log(|\fff|^2+\epsilon)\big)^k\w\mu,
\quad k=0,1,2, \ldots,
\end{equation*}
exist and are generalized cycles with Zariski support on $|\mu|$. 
The generalized cycles
$$
M^{\fff}_k\w \mu:=\1_Z\big((dd^c\log|\fff|^2)^k\w\mu\big), \quad k=0,1,2, \ldots,
$$
have Zariski support on $Z\cap|\mu|$.
If $\mu\sim 0$, then $M^{\fff}_k\w\mu\sim 0$. If $g$ is a 
holomorphic section of another vector bundle
that also defines $\J$,  
then
$M^{\fff}_k\w\mu\sim M^{g}_k\w\mu$.
\end{thm}

\noindent In case $\mu=\1_X$ we write $M^\fff_k$ rather than
$M^{\fff}_k\1_X$.
We let $M^\sigma\wedge \mu= M^\sigma_0\wedge \mu + M^\sigma_1\wedge
\mu+ \cdots$. 
%
%

\begin{df}\label{skott}
Assume that $\J\to X$ is defined by the section $\fff$ of the Hermitian
vector bundle $E\to X$. Given $\mu\in \B(X)$ and a representative
$\hat \mu\in\GZ(X)$, we let the $\B$-Segre class
$S_k(\J,\mu)$ be the class in $\B(X)$ defined by
$M_k^{\fff}\w\hat\mu.$  We let $S(\J,\mu)=S_0(\J,\mu)+S_1(\J,\mu)+\cdots.$
\end{df}
  
Notice that $M_k^{\fff}\w\hat\mu$ has support in $Z\cap |\mu|$ so that we may identify 
$S(\J,\mu)$ with an element in $\B(Z\cap |\mu|)$, in $\B(Z)$, or in $\B(|\mu|)$. 
If $\mu=\1_X$ we denote $S(\J, \mu)$ by $S(\J, X)$.

\begin{remark}\label{snigel}
If $\kappa=\max(0, \dim\mu-\dim (Z\cap |\mu|))$, 
then 
$$
S(\J, \mu)=S_\kappa(\J,\mu)+S_{\kappa+1}(\J,\mu)+\cdots + S_{\dim\mu}(\J,\mu).
$$
Indeed, $S_\ell(\J,\mu)$ has dimension $\text{dim}\,\mu-\ell$ and Zariski support $Z\cap |\mu|$, so 
$S_\ell(\J,\mu)=0$ if
$\text{dim}\, Z\cap\mu < \text{dim}\,\mu -\ell$ by the dimension principle. Moreover, 
clearly $S_\ell(\J,\mu)=0$ for degree reasons if $\ell>\dim \mu$. 

If $\J$ vanishes identically on $|\mu|$,  then it follows from the definition that 
$S(\J,\mu)=\mu$.
\end{remark}

One can define $M_k^{\fff}\w\mu$ by a limit procedure 
without applying $\1_Z$,
see \cite[Proposition~5.7 and Remark~5.9]{aeswy1}:



\begin{prop}\label{struts}
Let $\fff$ be a holomorphic section of a Hermitian bundle $E\to X$ and let
$$ 
M_{k,\epsilon}^{\fff}=\frac{\epsilon}{(|\fff|^2+\epsilon)^{k+1}}
(dd^c|\fff|^2)^k, \quad k=0,1,2,\ldots.
$$ 
If $\mu\in\GZ(X)$, then 
\begin{equation}\label{epsilon}
M_k^{\fff}\w\mu=\lim_{\epsilon\to 0} M_{k,\epsilon}^{\fff}\w\mu,
\quad k=0,1,2,\ldots.
\end{equation}
Moreover,  $M^{\fff}\w\mu=\sum_kM^\fff_k\wedge\mu$ is the value at $\lambda=0$, via analytic continuation from $\Re\lambda\gg 0$,
of 
$$
M^{\fff,\lambda}\w\mu=\Big(1-|\fff|^{2\lambda}+\sum_{k=1}\dbar|\fff|^{2\lambda}\w\frac{\partial|\fff|^2}{2\pi i|\fff|^2}
\w(dd^c\log|\fff|^2)^{k-1}\Big)\w\mu.
$$
\end{prop}

\begin{ex}
If $\mu\in \GZ(X)$ and $\gamma\wedge\mu\in\GZ(\U)$, where $\U\subset X$ is open and
$\gamma$ is a smooth form in $\U$, then by \eqref{epsilon}
\begin{equation}\label{lamm}
M^\fff\w(\gamma\w\mu)=\gamma\w M^\fff\w\mu
\end{equation}
in $\U$. 
%
\end{ex}

\begin{ex}  
If $f\colon X' \to X$ is proper, $\mu'\in\GZ(X')$, and $\mu=f_*\mu'$, then
\eqref{ingen} and  \eqref{epsilon} imply that
\begin{equation}\label{ponny}
M^{\fff}\w\mu=f_* (M^{f^*\fff}\w\mu').
\end{equation}
\end{ex}

Let $\xi$ be a section of a vector bundle
in a \nbh $\U\subset X$ of $x$ such that 
$\xi$ defines the maximal ideal at $x$.
Notice that if $\mu\in\GZ_k(X)$, then by Theorem~\ref{putig},
$M^{\xi}\w\mu$ is a  generalized cycle with Zariski support at $x$
and its image in $\B(X)$ is independent of the choice of
section $\xi$  defining the maximal ideal. In view of 
the dimension principle, see Section~\ref{potta}, 
$M^{\xi}\w\mu=M^{\xi}_k\w \mu=a[x]$ for some real 
number $a$.   
We say that $a$ is the {\it multiplicity},
$\mult_x\mu$, of $\mu$ at $x$, i.e.,
\begin{equation}\label{plutt}
\mult_x \mu=\int_\U  M^{\xi}\w\mu. 
\end{equation}
It is an integer that is independent of the choice of \nbh $\U$
and   
only depends on the class of $\mu$ in $\B(X)$.
If $\mu$ is effective (i.e., represented by a positive current), then $\mult_x\mu$ is  the Lelong number
of $\mu$ at $x$ and hence nonnegative, see \cite[Section~6]{aeswy1}.  

\begin{ex}\label{predikat} 
If $\mu\in \GZ(X)$ is of the form $\mu=\gamma\w\mu'$ in a neighborhood
of $x$, where
$\gamma$ is a closed smooth form of positive degree and $\mu'\in\GZ(X)$, then
$\mult_x\mu=0$. In fact, by \eqref{lamm}, $M^\xi\w\mu=\gamma\w M^\xi\w\mu'$ which must vanish
by the dimension principle, since $M^\xi\w\mu'$ has support at $x$ and $\gamma$ has positive degree.
\end{ex}

\subsection{Segre numbers}\label{snumbers}
Let $\J\to X$ be a coherent ideal sheaf over $X$ of codimension $p$. 
In  \cite{Twor} and \cite{GG} Tworzewski,  and Gaffney and Gassler, independently introduced,
at each point $x\in X$,  a list of numbers $(e_p(\J,X,x),\ldots, e_n(\J,X,x))$, called Segre numbers
in \cite{GG}. The Segre numbers generalize the Hilbert-Samuel multiplicity
at $x$ in the sense that if $\J$ has codimension $n$ at $x$ then $e_n(\J,X,x)$ is the Hilbert-Samuel multiplicity at $x$.  
The definitions in \cite{Twor} and \cite{GG},
though slightly different, are both of geometric nature.
There is also a purely algebraic definition, \cite{AM,AR}.  In \cite{aswy} were introduced semi-global currents
whose Lelong numbers are precisely the Segre numbers. 
These currents are generalized cycles where they are defined.  

We can define Segre numbers for  $\J$ over a  generalized cycle $\mu\in\GZ(X)$:
In a \nbh $\U$ of a given point $x$ we can take a section $\fff$ of a trivial Hermitian bundle such that
$\fff$ generates $\J$ and define the \emph{Segre numbers} 
$$
e_k(\J,\mu, x):=\mult_x (M_k^{\fff}\w\mu), \quad k=\kappa, \ldots, \dim
\mu, 
$$
where $\kappa$ is as in Remark \ref{snigel}. 
In view of Theorem~\ref{putig}, these numbers are independent of the choice 
of \nbh  $\U$ and of section $\fff$, and only depend on the class of $\mu$ in $\B(X)$.   
If $\mu=\1_X$, then $e_k(\J,\mu,x)$ coincides with $e_k(\J, X, x)$,
see \cite[Theorem~1.1]{aswy}.

\subsection{Regular embeddings and Gysin mappings}\label{guleboj}

Assume now that $X$ is smooth and that $\J\to X$ is locally a complete intersection
of codimension $\kappa$. This means that  
$\iota\colon Z_\J\hookrightarrow X$ is a {\it regular embedding}, where 
$Z_\J$ is the non-reduced space of codimension $\kappa$ defined by $\J$.
Then the normal cone $N_\J X$ is a vector bundle over the reduced space
$i\colon Z\hookrightarrow X$ and hence there is a well-defined cohomology class
$c(N_\J X)$ on $Z$. 
Therefore there is a well-defined mapping, the classical  
\emph{Gysin mapping}\footnote{Since this is a map to $\A_{k-\kappa}(Z)$, to be formally correct,
we must insert $i_*$ in the formula defining $\iota^!$, cf.\ Section~\ref{byggnad}.}
\begin{equation}\label{agysin}
\iota^!\colon \A_k(X)\to \A_{k-\kappa}(Z),\quad 
i_*\iota^!\mu = \big(c(N_\J X) \cap s(\J, \mu)\big)_{k-\kappa},
\end{equation}
where the lower index $k-\kappa$ denotes the component of dimension $k-\kappa$.
We have the analogous $\B$-\emph{Gysin mapping} 
\begin{equation}\label{bgysin}
\iota^!\colon\B_k(X)\to \B_{k-\kappa}(Z),\quad 
i_*\iota^!\mu = \big(c(N_\J X)\w S(\J, \mu)\big)_{k-\kappa}.
\end{equation}
Our main interest is when $\J$ defines a submanifold; in this case $Z=Z_\J$
and $i=\iota$. 

By suitable choices we can represent \eqref{bgysin} by a mapping on $\GZ(X)$:
Assume that $\J$ is defined by a section $\fff$ of a Hermitian
vector bundle $E\to X$ and let
$E'$ be the pull-back to $Z$. There is a canonical holomorphic embedding
$\varphi\colon N_\J X\to E'$,  
see \cite[Section~7]{aeswy1}. 
Let us equip
$N_\J X$ with the induced Hermitian metric and let $\hat c(N_\J X)$ be 
the associated Chern form, cf.~Section~\ref{seg}. Then we have the concrete mapping 
\begin{equation*}
\iota^!\colon\GZ_k(X)\to \GZ_{k-\kappa}(Z),\quad 
i_*\iota^!\mu =\big(\hat c(N_\J X)\w M^\fff\w \mu\big)_{k-\kappa}
\end{equation*}
which induces the mapping  \eqref{bgysin}.
We recall \cite[Propositions~1.4 and 1.5]{aeswy1}:
\begin{prop}\label{gata1}
If $\J\to X$ defines a regular embedding, then
$$
S(\J,X)=s(N_\J X)\w[Z_\J], \quad S_k(\J,X)=s_{k-\kappa}(N_\J X)\w[Z_\J]
$$
in $\B(X)$, where $[Z_\J]$ is (the Lelong current of) the fundamental
cycle associated to $\J$. 
If $\sigma$ defines $\J$, then 
$$
M^\fff= \hat s(N_\J X)\w[Z_\J], \quad M_k^\fff= \hat s_{k-\kappa}(N_\J X)\w[Z_\J]
$$
in $\GZ(X)$. 
\end{prop}

\begin{ex}
Let $i\colon Z\to X$ be the inclusion of a smooth submanifold of codimension $\kappa$ and suppose that $\mu\in\GZ_k(X)$ is a smooth
form. Then, in view of Proposition~\ref{gata1},
\begin{equation*}
i_*i^!\mu=\big(\hat c(N_ZX)\wedge\hat s(N_ZX)\wedge [Z]\wedge\mu\big)_{k-\kappa}=
[Z]\wedge\mu.
\end{equation*}
Thus, $i^!\mu=i^*\mu$ is the usual pullback.
\end{ex}

 
\subsection{Intersection with divisors and the Poincar\'e-Lelong formula on a generalized
  cycle}\label{strutta}

See \cite[Section~8]{aeswy1} for proofs of the statements in this subsection. 
Let $h$ be a meromorphic section of a line bundle $L\to X$. We say that
$\div h$ intersects the generalized cycle $\mu$ \emph{properly} if $h$
is generically holomorphic and nonvanishing on the Zariski support  $|\mu_j|$ of each irreducible component $\mu_j$ of $\mu$. 
If $\div h$ and $\mu$ intersect properly there
is a  generalized cycle $\div h\cdot \mu$ 
with Zariski support on $|\div h|\cap |\mu|$ that we call {\it the proper intersection}
of $\div h$ and $\mu$.  

If $\tau\colon W\to X$ such that $\mu=\tau_*\alpha$, where $\alpha$ is a product of components of
Chern or Segre forms, then 
$
\div h\cdot \mu= \tau_*([\div \tau^*h]\w \alpha).
$  
Then $\div h\cdot \mu\sim 0$ if $\mu\sim 0$ so that the intersection has meaning for  $\mu\in\B(Y)$.  
If $h$ is holomorphic, i.e., $\div h$ is effective, then, in a local frame for $L$,
\begin{equation}\label{hare}
\div h\cdot \mu =dd^c(\log|h|_\s^2~\mu)=\lim_{\epsilon\to 0}\big(dd^c
\log(|h|_\s^2+\epsilon)\w\mu\big), 
\end{equation}
where $|h|_\s$ is the norm of the holomorphic function obtained from any fixed local frame for $L$ so that
$dd^c\log |h|_\s$ is well-defined.  
It follows that $\div h\cdot \mu$ is effective if both $\div h$ and $\mu$ are effective. 
In light of \eqref{hare} it is natural to write $\div h\cdot \mu$ as
$[\div h]\w\mu$.

\begin{prop}[The Poincar\'e-Lelong formula on a generalized cycle]\label{myrstack}
Let $h$ be a nontrivial meromorphic section of a Hermitian line
bundle $L\to X$. Assume that $\div h$ intersects $\mu$ properly. 
Then
\begin{equation*}
dd^c(\log |h|^2 \mu)=[\div h]\w\mu-\hat c_1(L)\w\mu.
\end{equation*}
\end{prop}
 
\begin{remark}\label{luminy} 
If $\div h$ does not intersect $\mu$ properly we define  
$
[\div h]\wedge\mu=\sum_j [\div h]\wedge \mu'_j,
$
where $\mu'_j$ are the irreducible components of $\mu$ that $\div h$
intersects properly, see \cite[Section~9]{aeswy1}. 
\end{remark}

\subsection{Mappings into cohomology groups}
In this subsection we assume
that $X$ is projective, in particular compact, cf.~\cite[Section~10]{aeswy1}.
Let $\widehat H^{k,k}(X)$ be the equivalence classes of 
$d$-closed $(k,k)$-currents $\mu$
on $X$ of order zero such that $\mu\sim 0$ if there is a current
$\gamma$ of order zero such that $\mu=d\gamma$. 
If $X$ is smooth there is a natural isomorphism $\widehat H^{n-k,n-k}(X)\to H^{n-k,n-k}(X,\C)$; 
the surjectivity
is clear and the injectivity follows since a closed current of order zero
locally has a potential of order zero. 
If $i\colon X\hookrightarrow  M$ is an embedding into a smooth
manifold $M$ of dimension $N$, then there is a natural mapping
$i_*\colon \widehat H^{n-k,n-k}(X)\to H^{N-k,N-k}(M,\C)$ induced by the push-forward of currents.

There are natural cycle class mappings
\begin{equation}\label{Adef}
A_X\colon \A_k(X)\to \widehat H^{n-k,n-k}(X), \quad k=0,1, \ldots, 
\end{equation}
and,  \cite[Eq.~(10.8)]{aeswy1}, 
\begin{equation*}
A_X (c(E)\cap \mu)=c(E)\w A_X\mu,
\end{equation*}
in $\widehat H(X)$, where the right hand side is represented by the wedge product of a smooth form and a current.  
There are natural mappings   
\begin{equation}\label{Bdef}
B_X\colon \B_k(X)\to \widehat H^{n-k,n-k}(X), \quad k=0,1, \ldots,
\end{equation}
and clearly  $B_X (c(E)\wedge \mu)=c(E)\w B_X\mu$.

\begin{ex}\label{dota} 
Assume that $h$ is a meromorphic section of a Hermitian line bundle
$L\to X$ such that $\div h$ intersects $\mu\in\GZ_k(X)$ properly.  It follows from
Proposition~\ref{myrstack} that $[\div h]\w\mu$ and $\hat c_1(L)\w\mu$ coincide in 
$\widehat H^{n-k+1,n-k+1}(X)$. 
\end{ex}

Let us recall, \cite[Proposition~1.6]{aeswy1}, that the images of $\A_k(X)$ and $\B_k(X)$
in $\widehat H^{n-k,n-k}(X)$ coincide. We have the
commutative diagram
\begin{equation*}
\begin{array}[c]{ccc}
\Zy_k(X) & {\hookrightarrow} & \B_k(X) \\
\downarrow & &  \downarrow \scriptstyle{B_X} \\
\A_k(X) & \stackrel{A_X}{\longrightarrow} &  \widehat H^{n-k,n-k}(X)
\end{array}. 
\end{equation*}

\begin{ex}\label{lattjo}
It follows from the dimension principle that  $\A_n(X)=\Zy_n(X)=\B_n(X)$.
If $X$ has the irreducible components $X_1, X_2, \ldots$, then 
the image in $\widehat H^{0,0}(X)$ of the cycle
$a_1X_1+a_2X_2+\cdots$ on $X$ is the $d$-closed $(0,0)$-current
$a_1\1_{X_1}+a_2\1_{X_2}+\cdots$.  It follows that the mappings
into $\widehat H^{0,0}(X)$ are injective.
\end{ex}

More generally, we have \cite[Proposition~1.7]{aeswy1}:
 
\begin{prop}\label{ab0}
Assume that $\J\to X$ defines a regular embedding $Z_\J\hookrightarrow X$
of codimension $\kappa$ and let $\mu$ be a cycle.  The images in $\widehat H^{*,*}(Z)$ 
of the Gysin and the $\B$-Gysin mappings of $\mu$, \eqref{agysin} and \eqref{bgysin}, 
coincide.
\end{prop}

\section{Local intersection numbers}\label{segretal}
Let  $Y$ be a smooth manifold, let $\mu_1, \ldots,\mu_r$ be generalized cycles on $Y$ of pure dimensions  
and let 
$d=\dim \mu_1+\cdots+\dim\mu_r$.
Following the ideas of Tworzewski \cite{Twor} we define 
the {\it local intersection numbers} at $x$, cf.~Lemma~\ref{skaldjur} and~Section~\ref{snumbers},  
\begin{equation*}
\epsilon_\ell(\mu_1,\ldots,\mu_r, x):=e_{d-\ell}\big(\J_\Delta,\mu_1\times \cdots \times \mu_r,i(x)\big), \quad \ell=0,1,\ldots, d,
\end{equation*}
where 
$i\colon Y\hookrightarrow Y^r:=Y\times\cdots\times Y$ is the
parametrization $x\mapsto (x,\ldots, x)$ 
of the diagonal $\Delta$ in  $Y^r$ 
and  $\J_\Delta\to  Y^r$ 
is the ideal sheaf that defines
$\Delta$. 
Notice that if $E\to  Y\times\cdots\times Y$ is a Hermitian vector bundle and 
$\fff$ is a section of $E$ that generates $\J_\Delta$, then
$M^\fff\w (\mu_1\times \cdots \times \mu_r)$ is a global generalized cycle such that
\begin{equation}\label{lok}
\epsilon_\ell(\mu_1,\ldots,\mu_r, x)=
\mult_{i(x)} M_{d-\ell}^\fff\w (\mu_1\times \cdots \times \mu_r)
\end{equation}
for 
$\ell\le d$.  
More invariantly we have, cf.~Definition~\ref{skott}, 
\begin{equation}\label{blomma}
\epsilon_\ell(\mu_1,\ldots,\mu_r, x)=
\mult_{i(x)} S_{d-\ell}(\J_\Delta, \mu_1\times \cdots \times \mu_r).
\end{equation}

Given a point $x$, \eqref{lok} holds as soon as $\fff$ defines $\J_\Delta$ in a \nbh of the point $i(x)$
so we can assume that $\fff$ is a section of a trivial bundle. If the $\mu_j$ are cycles, therefore these numbers 
coincide with the local intersection numbers \eqref{locsec} introduced
by Tworzewski in \cite{Twor},  cf.\ Section~\ref{snumbers} and \cite[Section~10]{aswy}.

\begin{remark}\label{tprod}
Tworzewski, \cite{Twor},  proved that there is a unique global cycle $\mu$ 
such that the sum of its multiplicities, of its components of various dimensions,
at each point $x\in V$ coincides with the sum of the local intersection numbers at $x$. 
Since this definition is local, it cannot carry global information.
For instance, the self-intersection, in this sense, of any smooth curve $Z$ in $\P^2$
is just the curve itself, and therefore the B{\'e}zout formula, cf.~\eqref{likhet}, is not satisfied
unless $Z$ is a line.
\end{remark}

\section{The $\B$-St\"uckrad-Vogel class in $\Pr^M$}\label{vogelsec} 
Let $\Pr^M$ be the projectivization of $\C_{x_0,\ldots,x_M}^{M+1}$. 
Let $\eta=(\eta_1,\ldots,\eta_m)$ be a tuple of linear forms on $\C^{M+1}$ in general
position. As usual we identify the $\eta_j$ with sections of the line bundle
 $L=\Ok(1)\to \Pr^M$ and 
$\eta$ with a section of $E:=\oplus_1^mL$. 
Similarly to Section~\ref{strutta} we let $|\eta|_\s$ be the norm of the holomorphic 
tuple obtained from any fixed local frame for $L$ so that
$dd^c\log |\eta|_\s$ is well-defined.
 Let $Z$ be the plane of codimension $m$
 that $\eta$ defines and let $\J\to \P^M$ be the associated radical ideal sheaf.

Let $\mu$ be a fixed generalized cycle in $\Pr^M$ of pure dimension $d$.
For a generic choice of $a=(a_1,\ldots, a_d)\in (\in\Pr^{m-1})^d$, the successive intersections\footnote{We let $\1_Z$ as 
well as $[\div (a_j\cdot \eta)]$ act on the whole current on
its right, i.e., $\1_Z\gamma\w\mu:= \1_Z(\gamma\w\mu)$ etc.}
by divisors,
cf.~Section~\ref{strutta}, in
\begin{equation}\label{badanka}
v_k^{a\cdot\eta}\wedge\mu:= \1_Z[\div(a_k\cdot \eta)]\w\1_{X\setminus Z} [\div(a_{k-1}\cdot \eta)]\cdots \w\1_{X\setminus Z} 
[\div(a_1\cdot \eta)]\w\1_{X\setminus Z}\mu
 \end{equation}
for $k=0,\ldots,d$  are proper, 
and
\begin{equation}\label{godis}
v^{a\cdot\eta}\wedge\mu=\sum_{k=0}^d v_k^{a\cdot\eta}\wedge\mu
\end{equation}
is the resulting  \emph{St\"uckrad-Vogel (SV) cycle},  cf.\
\cite[Section~9]{aeswy1}.  

\begin{prop}
If we take the mean value of  \eqref{godis} over $(\P^{m-1})^d$, with respect to normalized Haar measure,
then we get the generalized cycle
\begin{equation}\label{helig}
M^{L,\eta}\w\mu:={\bf 1}_Z\mu +{\bf 1}_Zdd^c\log|\eta|^2_\s\w\mu+\cdots +
{\bf 1}_Z(dd^c\log|\eta|^2_\s)^d\w\mu.
\end{equation}
\end{prop}

\begin{proof}  With the convention in Remark~\ref{luminy} we can write 
$$
v_k^{a\cdot\eta}\wedge\mu=\1_Z[\div(a_k\cdot \eta)]\w [\div(a_{k-1}\cdot \eta)]\w\cdots \w 
[\div(a_1\cdot \eta)]\w\mu.
$$
Now the proposition follows from \cite[Proposition~9.3]{aeswy1}.
\end{proof}
%

By \cite[Proposition~9.5]{aeswy1}, the class of $M^{L,\eta}\w\mu$ in $\B(\P^M)$ only depends on 
$\J$, $L$, and $\mu$ and not on the choice of generators $\eta$. 

\begin{df}
For $\mu\in\B(\P^M)$, we let $V(\J,L,\mu)$, the $\B$-\emph{SV-class} of $L$ and $\J$ on $\mu$, 
be the class of $M^{L,\eta}\w\mu$ in $\B(\P^M)$.
\end{df}

Notice that $M^{L,\eta}\w\mu$ has support in $Z\cap |\mu|$ so that we may identify $V(\J,L,\mu)$
with an element in $\B(Z\cap |\mu|)$,
cf.~\cite[Definition~9.6]{aeswy1}.

\smallskip

Let $\U\subset\P^M$ be an open set where we have a local frame $e$ for $L$. For
instance, each nontrivial section of $L$ vanishes on a hyperplane $H$ 
and thus gives rise to a local frame in the open set $\Pr^M\setminus H$. 
In $\U$ we have that 
\begin{equation}\label{ecuador}
M^{L,\eta}\wedge\mu=M^\eta\wedge\mu
\end{equation}
with the metric on $L|_\U$ such that $|e|=1$, cf.\ \cite[Remark~8.2]{aeswy1}. 
It follows that local statements that hold for $M^\eta\wedge\mu$ must
hold for $M^{L,\eta}\wedge\mu$ as well. In particular, if $\eta$
defines the maximal ideal at $x\in \P^M$, then, in view of
\eqref{plutt}, 
\begin{equation}\label{skrutt}
M^{L,\eta}\wedge\mu =\mult_x \mu \cdot [x]. 
\end{equation}
%
By  \eqref{epsilon} and \eqref{ecuador}, in $\U$ we have the regularization
\begin{equation}\label{epsilon2}
M_k^{L,\eta}\w\mu=\lim_{\epsilon\to 0}\frac{\epsilon}{(|\eta|^2_\s+\epsilon)^{k+1}}
(dd^c|\eta|^2_\s)^k \w\mu, \quad k=0,1,2,\ldots.
\end{equation}
In particular, $M_k^{L,\eta}\w\mu$ is
effective if $\mu$ is; indeed $dd^c|\eta|^2_\s$ is a positive
$(1,1)$-form.

We have the Fubini-Study norm
$|\xi|=\|\xi\| /\| x\|$ on $L=\mathcal{O}(1)$,  where $\|\cdot \|$ denotes the Euclidean norm on $\C^{M+1}_x$.
 
\begin{prop}\label{struts2}
With the norm above $M^{L,\eta}\w\mu$ is the value at $\lambda=0$ of the current valued function 
\begin{equation}\label{lambdasond}
\lambda\mapsto \Big(1-|\eta|^{2\lambda} +\sum_{k\ge 1}\frac{\dbar|\eta|^{2\lambda}\w\partial |\eta|^2}{2\pi i|\eta|^2}
\w (dd^c\log|\eta|^2_\s)^{k-1}\Big)\w\mu,
\end{equation}
a~priori defined when  $\Re\lambda\gg 0$. 
\end{prop}

\begin{proof} 
The statement follows directly from Proposition~\ref{struts} in a 
set where we have a local frame for $L$ if we replace 
each occurrence of $|\eta|$ in \eqref{lambdasond} by
$|\eta|_\s$. However one can verify, cf.\ \cite[proof of Lemma~2.1]{ABullSci}, that
the value at $\lambda=0$ is independent of the choice of norm on $L$, and thus 
the proposition follows.
\end{proof}

Notice that the Fubini-Study form  $\hat\omega=dd^c\log|x|^2_\s=dd^c\log\|x\|^2$ represents
the first Chern class $\omega=c_1(L)$. 
We have van~Gastel's formulas for generalized cycles, \cite[Theorem~9.7]{aeswy1}, 
\begin{equation}\label{gastel12}
M^{L,\eta}\w \mu=\sum_{j\ge 0}\Big(\frac{1}{1-\hat\omega}\Big)^j\w M^{\eta}_j\w\mu
\end{equation}
and
\begin{equation}\label{gastel22}
M^{\eta}\w\mu=\sum_{j\ge 0}\Big(\frac{1}{1+\hat\omega}\Big)^j\w M^{L,\eta}_j  \w\mu.
\end{equation}
From \cite[Proposition~9.12]{aeswy1} we get, cf. \eqref{degdef}, 

\begin{prop}\label{skottradd}
Assume that  $\mu\in\GZ_d(X)$. 
We have the mass formula
\begin{equation}\label{B{\'e}zout2}
\deg \mu=\deg M^{L,\eta}_0\w\mu+\cdots +\deg  M^{L,\eta}_d\w\mu +
\deg \big(\ett_{X\setminus Z}(dd^c\log|\eta|_\s^2)^d\w\mu\big).
\end{equation}
\end{prop}
 

\noindent If $m\le d$, then the last term in \eqref{B{\'e}zout2}  vanishes since
$(dd^c\log|\eta|_\s^2)^m=0$ outside $Z$.  

\smallskip
For future reference we also point out the following invariance result.
Assume that $i\colon \Pr^M\to\Pr^{M'}$ is a linear embedding of $\Pr^M$ in $\Pr^{M'}$.  
Let $p\colon \Pr^{M'} \dashrightarrow\Pr^M$ 
be a projective (generically defined) projection, i.e., induced by an affine projection $\C^{M'+1}\to\C^{M+1}$,
so that $p\circ i$ is the identity on $\Pr^M$.
Then $p^*\eta_j$ are well-defined linear forms on $\Pr^{M'}$. Let $\eta'$ be some additional linear forms
on $\Pr^{M'}$ that vanish on $i(\Pr^{M})$.  

\begin{prop}\label{pucko}
If $\mu\in\GZ(\Pr^M)$, then 
$$
M^{L, (p^*\eta,\eta')}\w i_*\mu=i_*(M^{L,\eta}\w\mu).
$$
\end{prop}

\begin{proof} 
Since  $\eta'=0$ on the Zariski support of $i_*\mu$,
$M^{L,(p^*\eta,\eta')}\w i_*\mu= M^{L,(p^*\eta,0)}\w i_*\mu$.
Now the proposition follows from \eqref{ingen} and Proposition~\ref{struts2}, or \eqref{epsilon2}, 
since $\eta=i^*p^*\eta$. 
%
\end{proof}

\section{$\B$-intersection products on manifolds}\label{intprod}
Assume that $\mu_1, \ldots, \mu_r$ are cycles on a complex
manifold $Y$ of dimension $n$ as in the introduction. It is well-known that if they intersect properly, then, 
see, e.g.,  \cite[Chapter~12]{Ch}, one can define the wedge product
$[\mu_1]\wedge\cdots\wedge[\mu_r]$ by means of appropriate
regularizations, see, e.g., \cite[Chapter~III.3]{Dem2}, and this current coincides with
(the Lelong current of) the proper intersection cycle 
$\mu_1\cdot_Y\cdots\, \cdot_Y\mu_r$, see, e.g., \cite[page~212]{Ch}.
It is easy to see that the cycle  $\mu=\mu_1\times\cdots \times \mu_r$ and the diagonal $\Delta$ in 
$Y^r=Y\times \cdots \times Y$ 
intersect properly, and one can prove that if we identify $\Delta$ and $Y$, then 
the proper intersection $\Delta\cdot_{Y^r} \mu$ coincides with $\mu_1\cdot_Y\cdots \, \cdot_Y\mu_r$.
If the $\mu_j$ do not intersect properly the basic idea is to define the intersection of $\Delta$
and $\mu_1\times\cdots\times\mu_r$, 
cf.\ Section~\ref{segretal}.
The advantage then is that one of the factors is a regular embedding.

\smallskip
We now recall the classical nonproper intersection product. 
%
If $\iota\colon Z_\J\to Y$ is a regular embedding of codimension $\kappa$ and $\mu\in \A_k(Y)$, 
then we have, cf.\ \eqref{agysin}, the  product
\begin{equation}\label{Gysinprod}
Z_\J\diamond_{Y}\mu=\iota^!\mu,    
\end{equation}
see, e.g., \cite[Chapter~6.1]{Fult} for background and motivation. 
Let 
\begin{equation}\label{valkyria}
i\colon Y\hookrightarrow Y^r, \quad x\mapsto (x,\ldots, x),
\end{equation} 
be the  diagonal $\Delta$; notice that this is a
regular embedding. 
Given arbitrary cycles $\mu_1,\ldots, \mu_r$, we define the intersection product
\begin{equation*}
\mu_1\cdot_Y\cdots\cdot_Y\mu_r:=i^! (\mu_1\times\cdots\times\mu_r),   
\end{equation*} 
see, e.g., \cite[Chapter~8.1]{Fult}. After identification of $Y$ and $\Delta$
we have $\mu_1\cdot_Y\cdots\cdot_Y\mu_r=\Delta\diamond_{Y^r} (\mu_1\times\cdots\times\mu_r)$.
In case $\mu_1=Z_\J$  is a regular embedding and $\mu_2$ is an
arbitrary cycle, then $\mu_1\cdot_Y \mu_2 = \mu_1\diamond_Y \mu_2$ coincide, see \cite[Corollary~8.1.1]{Fult}.

\smallskip 
We will define analogues for $\B(Y)$, cf.\ Definition~\ref{skott},
Lemma~\ref{skaldjur}, and \eqref{hopper}. 
 
 
\begin{df}\label{strunta}
Assume that $\iota\colon Z_\J\to Y$ is a regular embedding.
For $\mu\in\B(Y)$ we define, cf.\ \eqref{bgysin}, 
the product
\begin{equation*}
Z_\J  \diamond_{\B(Y)}\mu=\iota^! \mu.   
\end{equation*}
\end{df}

Notice that if $Z_\J$ has codimension $\kappa$ and $\mu\in\B_k(Y)$,
then $Z_\J  \diamond_{\B(Y)}\mu\in\B_{k-\kappa}(Z)$; recall that $Z$ is the zero set of $\J$.
Moreover, the Zariski support of $Z_\J  \diamond_{\B(Y)}\mu$ is contained in $Z\cap |\mu|$ and so we can 
identify $Z_\J  \diamond_{\B(Y)}\mu$ with an element in $\B_{k-\kappa}(Z\cap|\mu|)$.

\begin{remark}
If $\J$ is the radical ideal of a submanifold or a reduced locally complete intersection $i\colon Z\hookrightarrow Y$
of codimension $\kappa$ and $\mu$ is a $k$-cycle in $Y$ intersecting $Z$ properly,
then $i_*(Z  \diamond_{\B(Y)}\mu)$ is the proper intersection
$[Z]\wedge \mu$.
%
In fact, in view of Definition~\ref{skott} and
Proposition~\ref{gata1}, 
\begin{equation*}
S(\J, \mu) = i_*S(i^* \J, \mu)=
i_* \big (s(N_{i*\J}\mu)\wedge [Z_{i^*\J}]\big ) 
=
s(N_\J Y)\wedge i_* [Z_{i^*\J}]=  s(N_\J Y)\wedge [Z]\wedge \mu. 
\end{equation*}
Thus, by \eqref{bgysin}, 
\[i_*(Z  \diamond_{\B(Y)}\mu)=i_*i^!\mu=
\big(c(N_\J Y)\wedge S(\J,\mu)\big)_{k-\kappa}=
\big(c(N_\J Y)\wedge s(N_\J Y) \big)_{0}\wedge [Z]\wedge \mu=
[Z]\wedge \mu.
\]
\end{remark}


\begin{df} 
If $\mu_1,\ldots,\mu_r$ are elements in $\B(Y)$, we define 
 \begin{equation*}
\mu_1\cdot_{\B(Y)}\cdots\cdot_{\B(Y)}\mu_r: = i^! (\mu_1\times\cdots\times\mu_r).
\end{equation*}
 \end{df}
 
As above, notice that after identification of $Y$ and $\Delta$ we have 
$\mu_1\cdot_{\B(Y)}\cdots\cdot_{\B(Y)}\mu_r = \Delta\diamond_{\B(Y^r)} \mu_1\times\cdots\times\mu_r$

\begin{remark}\label{sodavatten}
Let $p\colon Y^r\to Y$ be the projection on one of the factors. Then $p\circ i=id$, hence
$p_* i_*=id$ and thus
$
\mu_1\cdot_{\B(Y)}\cdots\cdot_{\B(Y)}\mu_r=p_*(\Delta\diamond_{\B(Y^r)} \mu_1\times\cdots\times\mu_r).
$
\end{remark}

Assume that $\mu_1$ is a regular embedding.  Contrary to the classical intersection product case it is {\it not} true in general that 
$\mu_1\diamond_{\B(Y)}\mu_2$ and $\mu_1\cdot_{\B(Y)}\mu_2$ coincide.    
Example~\ref{cusp2} below shows that the $\B$-self-intersection of the cusp $\mu=\{x_1^3-x_0x_2^2=0\}\subset  \Pr^2$
is different from $\mu\diamond_{\B(Y)} \mu$. This example also shows that the $\B$-analogue of the classical
self-intersection formula  does not hold in general. However, it is true for smooth cycles.

\begin{prop}[Self-intersection formula]\label{sparrow}
Let $V\hookrightarrow Y$ be a smooth subvariety of $Y$ of codimension $m$.  Then
\begin{equation}\label{sparv3}
V\cdot_{\B(Y)} V=c_m(N_V Y)\w [V]. 
\end{equation}
\end{prop}

\begin{proof}
Notice that the diagonal $\Delta_Y$ is smooth in $Y\times Y$
and that $N_{\Delta_Y}(Y\times Y)=T{\Delta_Y}$.
If $j\colon V\times V\to Y\times Y$ is the product embedding, then
$
j^*\J_{\Delta_Y}=\J_{\Delta_{V}}.
$
Therefore $i_*(V\cdot_{\B(Y)} V)=\Delta_Y\diamond_{\B(Y\times Y)} V\times V$ 
is the component of dimension $n-2m$ of
$$
c\big(N_{\Delta_Y}(Y\times Y)\big )\w S(j^*\J_{\Delta_Y},V\times V)=
c(T\Delta_Y)\w S(\J_{\Delta_{V}},V\times V)=c(T\Delta_Y)\w
s(T\Delta_V)\w [\Delta_V],
$$
where the last equality follows from Proposition~\ref{gata1}
and, since $V$ is smooth, that $N_{\Delta_V}(V\times V)=T\Delta_V$.
Via the natural isomorphisms $Y\simeq{\Delta_Y}$ and $V\simeq \Delta_V$ thus $V\cdot_{\B(Y)} V$
is the component of dimension $n-2m$ of
 $$
c(TY)|_V\w s(TV)\w [V]=c(TY)|_V\w \frac{1}{c(TV)}\w [V]=
c(TY/TV)|_V\w[V]= c(N_V Y)\w [V], 
$$ 
cf.\ \eqref{smal}. 
Thus we get \eqref{sparv3}. 
\end{proof}

\begin{ex} 
Let $E$ be the exceptional divisor of the blow-up $Y=Bl_a\P^2\to \P^2$
at a point 
$a\in \P^2$.  
Let $L_E\to Y$ be the line bundle with a section that defines
$E$. It follows from \eqref{sparv3} that 
$
E\cdot_{\B(Y)} E= c_1(L_E)\w[E].
$
Since $-c_1(L_E)$ is positive $E\cdot_{\B(Y)} E$ is negative,
which is expected in view of the classical self-intersection of $E$.
\end{ex}

\noindent We have always coincidence of the various intersection
products on cohomology level;
recall the mappings \eqref{Adef} and \eqref{Bdef}.

 \begin{prop}\label{stupa}
Assume that $\mu_1, \ldots, \mu_r$ are cycles in $Y$ and let
$V=|\mu_1|\cap\cdots\cap|\mu_r|$.  Then \begin{equation}\label{sparv1}
A_V\big(\mu_1\cdot_{Y} \cdots\cdot_Y\mu_r\big)=B_V  \big(\mu_1\cdot_{\B(Y)} \cdots\cdot_{\B(Y)}\mu_r\big)
\end{equation}
in $\widehat H(V)$.
Moreover, if  $r=2$ and $\mu_1$ is a regular embedding, then
\begin{equation}\label{sparv2}
B_V \big(\mu_1\diamond_{\B(Y)} \mu_2\big)=B_V\big(\mu_1\cdot_{\B(Y)} \mu_2\big). 
\end{equation}
\end{prop}

\begin{proof}
The equality \eqref{sparv1} follows directly from the definitions and Proposition~\ref{ab0}.
Since the two possible definitions of $ \mu_1\cdot_{Y} \mu_2$ coincide when
$\mu_1$ is a regular embedding, \eqref{sparv2} follows by another
application of Proposition~\ref{ab0}. 
\end{proof}

 

\begin{prop} \label{gastuba}
(i)  \  If $\mu_1, \ldots, \mu_r$ are cycles in $Y$ that intersect properly, then
 \begin{equation}\label{sugga6}
 \mu_1\cdot_{\B(Y)}\cdots\cdot_{\B(Y)} \mu_r=\mu_1\cdot_Y\cdots\cdot_Y\mu_r.
 \end{equation}
%
%

\smallskip
\noindent (ii)   If $h$ is a holomorphic section of $L\to Y$ such that $\div h$ intersects $\mu\in\B(Y)$ properly, then
\begin{equation}\label{tryffel}
\div h \diamond_{\B(Y)}\mu=\div h\cdot \mu=\div h \cdot_{\B(Y)} \mu.    
\end{equation}
%
\end{prop}

\begin{proof} 
Assume that the $\mu_j$ have dimensions $d_j$, respectively.  
The assumption about proper intersection means that the set-theoretic 
intersection $V=|\mu_1|\cap\cdots\cap|\mu_r|$ has the expected
dimension $k:=d_1+\cdots + d_r-(r-1)n$ and that
$ \mu_1\cdot_Y\cdots\cdot_Y\mu_r$ and $\mu_1\cdot_{\B(Y)}\cdots\cdot_{\B(Y)} \mu_r$ are elements
in $\A_k(V)$ and $\B_k(V)$, respectively. Now \eqref{sugga6} follows from 
\eqref{sparv1} and Example~\ref{lattjo}.



Let us now consider part (ii). We may assume that $\mu=\tau_*\alpha$, where $\tau\colon W\to Y$
is proper holomorphic and $\alpha$ is a product of components of Chern or Segre forms, cf.\ \eqref{seedan}.
 The assumption of proper intersection implies that  $h$ is not identically zero on $|\mu|=\tau(W)$ so that 
$M^{h}_0\w\mu=\mathbf{1}_{h=0}\mu=\tau_*\mathbf{1}_{\tau^*h=0}\alpha=0$. Let $\iota$ be the regular embedding 
given by the ideal sheaf $\J_h$ generated by $h$.
We have $N_{\J_h}Y=L|_{h=0}$, cf.\ Section~\ref{guleboj}.
Thus 
\begin{eqnarray*}
\div h\diamond_{\B(Y)}\mu &=& \iota^!\mu=(c(L)\w S(\J_h,\mu))_{\dim\mu-1}=(c(L)\w M^{h}\w\mu)_{\dim\mu-1} \\
&=&
c_0(L)\wedge M_1^h\wedge \mu 
=M_1^h\wedge\mu 
=\div h\cdot \mu;
\end{eqnarray*}
for the last equality, cf.~\cite[Eq.~(8.4)]{aeswy1}.
%

We now consider the last equality in \eqref{tryffel}.
Consider the commutative diagram 
\begin{equation}\label{tapir}
\begin{array}[c]{ccc}
Y\times W  & \stackrel{id\times \tau}{\longrightarrow} & Y\times Y \\
\downarrow \scriptstyle{\pi} & &  \downarrow \scriptstyle{p} \\
W & \stackrel{\tau}{\longrightarrow} &   Y,
\end{array}
\end{equation}
where $p$ is the projection on the first factor. 
By definition, cf.~Remark~\ref{sodavatten}, $\div h \cdot_{\B(Y)}\mu$ 
is $p_*$ of
\begin{eqnarray}\label{spadtag}
\Delta\diamond_{\B(Y\times Y)}(\div h\times\mu) &=&
\big(c(N_\Delta (Y\times Y))\w S(\J_\Delta, \div h\times\mu)\big)_{\dim\mu-1} \\
&=&
\sum_{\ell=0}^n c_{n-\ell}(N_\Delta (Y\times Y))\wedge S_\ell(\J_\Delta, \div h\times\mu). \nonumber
\end{eqnarray}
Recall that $S_\ell(\J_\Delta, \div h\times\mu) = M^\fff_\ell\wedge(\div h\times\mu)$
if $\fff$ is a section that defines $\Delta\subset Y\times Y$.
Now 
$\div h\times\mu=(id\times \tau)_*(\div h\times \alpha)$ so if $g=(id\times\tau)^*\fff$
we have,  cf.~\eqref{ponny} and \eqref{lamm},  
\begin{multline*}
M_\ell^\sigma\w(\div h\times\mu)=(id\times\tau)_* M_\ell^g\w(\div h\times\alpha)=\\
(id\times \tau)_*\big((1\times\alpha)\w M_\ell^g\w(\div h\times 1)\big)=
(id\times \tau)_*\big((1\times\alpha)\w M_\ell^g\w [\div(h\otimes 1)]\big). 
\end{multline*}
Notice that $g$ defines the graph $G$  of $\tau$ in $Y\times W$. Since $\div h$ and $\mu$ intersect properly, $\tau^*h$ is 
generically non-vanishing on $W$ and so $h\otimes 1$ is generically non-vanishing on $G$. Thus, $G$ and $\div (h\otimes 1)$ intersect
properly. The Zariski support of 
$M_\ell^g\w [\div(h\otimes 1)]$ is $G\cap \{h\otimes 1 =0\}$, which thus has dimension $\dim W-1$. Since
$M_\ell^g\w [\div(h\otimes 1)]$ has dimension $\dim W +n-\ell-1$ it follows from the dimension principle that 
$M_\ell^g\w [\div(h\otimes 1)]=0$ for $\ell<n$. Thus, $S_\ell(\J_\Delta, \div h\times\mu)=0$ for $\ell<n$ and
from \eqref{spadtag} we get
\begin{equation}\label{spadtag1}
\Delta\diamond_{\B(Y\times Y)}(\div h\times\mu) =
S_n(\J_\Delta, \div h\times\mu)
=(id\times \tau)_*\big((1\times\alpha)\w M_n^g\w [\div(h\otimes 1)]\big).
\end{equation}

To compute $M_n^g\w [\div(h\otimes 1)]$, notice that $g$ defines a regular embedding in $Y\times W$ of codimension $n$
and that, since $\dim (G\cap \{h\otimes 1 =0\})=\dim W-1$, 
the restriction of $g$ to $\div(h\otimes 1)$ defines a regular embedding in $\div(h\otimes 1)$ of codimension $n$. Thus, by 
\cite[Corollary~7.5]{aeswy1},
\begin{multline}\label{spadtag2}
M_n^g\w [\div(h\otimes 1)] = \big(M^g\wedge [\div(h\otimes 1)]\big)_{\dim W-1} \\
=\big(S(\J_g,Y\times W)\wedge[G]\wedge [\div(h\otimes 1)]\big)_{\dim W-1} \\
=S_0(\J_g,Y\times W)\wedge[G]\wedge [\div(h\otimes 1)]
=[G]\wedge [\div(h\otimes 1)],
\end{multline}
where $\J_g$ is the ideal sheaf generated by $g$.
Since \eqref{tapir} is commutative, \eqref{spadtag1} and \eqref{spadtag2} give
\begin{eqnarray}\label{strutsfarm}
p_*\big(\Delta\diamond_{\B(Y\times Y)}(\div h\times\mu) \big) &=& \tau_*\pi_*\big((1\times\alpha)\w
[\div(h\otimes 1)] \wedge [G]\big) \\
&=&
dd^c\tau_*\pi_*\big( (1\times \alpha)\w (\log|h|_\s^2\otimes 1)\w [G]\big), \nonumber
\end{eqnarray}
cf.~\eqref{hare}. 
Since  
$\pi_*\big((1\times \alpha)\w (\log|h|_\s^2\otimes 1) \w[G]\big)= \log|\tau^*h|_\s^2\alpha$,
by \eqref{strutsfarm} we get, cf.~\eqref{ingen} and \eqref{hare}, 
\begin{equation*}
p_*\big(\Delta\diamond_{\B(Y\times Y)}(\div h\times\mu) \big)
= dd^c\tau_*\big( \log|\tau^*h|_\s^2\alpha)=dd^c(\log|h|^2_\s\mu)=\div h\cdot \mu,
\end{equation*}
finishing the proof.
 \end{proof}

%

\section{The $\bullet$-product on $\P^n$}\label{bulletsection}
 
In this section we define the product \eqref{poker} of generalized cycles on $\P^n$ and prove 
Theorem~\ref{mainthm}. 
%
The first step is to define the join of two generalized cycles. 
For simplicity we first assume that $r=2$. The mapping
\begin{equation}\label{boljig}
\P^{2n+1}_{x,y}\stackrel{\p}{\dashrightarrow} \P^n_x\times\P^n_y, \quad [x,y]\mapsto ([x],[y]).
\end{equation} 
is well-defined outside the union of the two disjoint $n$-dimensional planes $x=0$ and
$y=0$, and it has surjective differential.
If $\mu_1,\mu_2\in\GZ(\P^n)$,  therefore $\p^*(\mu_1\times\mu_2)$ is a well-defined current outside
the indeterminacy set of $\p$.  We will see that
$\p^*(\mu_1\times\mu_2)$ extends in a natural way to a generalized cycle 
$\mu_1\times_J\mu_2$ on $\P^{2n+1}_{x,y}$.

\smallskip
Let $\pi\colon Bl\,\P^{2n+1}_{x,y}\to \P^{2n+1}_{x,y}$ be the blow-up of
$\P^{2n+1}_{x,y}$ along 
$\{x=0\}$ and $\{y=0\}$.   
Then we have 
\begin{equation}\label{svangd}
\xymatrix{
Bl\,\P^{2n+1}_{x,y} \ar[d]_{\pi} \ar[dr]^p & \\
\mathbb{P}^{2n+1}_{x,y} \ar@{-->}[r]_{\frak{p}} & \mathbb{P}^n_x\times \mathbb{P}^n_y,
}
\end{equation}
where  $p:= \p\circ\pi \colon Bl\,\P^{2n+1}_{x,y}\to \P^n_x\times\P^n_y$
has surjective differential and hence is smooth, i.e., maps smooth forms onto smooth forms.

\begin{lma} \label{lagom}
(i) If $\mu\in\GZ(\Pr^n\times\Pr^n)$, then $p^*\mu\in\GZ(Y)$.

(ii)  $\pi_*p^*\mu$ is in $\GZ(\Pr^{2n+1})$
and coincides with $\p^*\mu$ where it is defined.

(iii)  If $\mu=0$ in $\B(\Pr^n\times\Pr^n)$, then 
$\pi_*p^*\mu=0$ in $\B(\Pr^{2n+1})$.
\end{lma}

\begin{proof} 
Note that (ii) is a direct consequence of (i). 

Let $X=\P^n\times\P^n$ and $X'=Bl\,\P^{2n+1}_{x,y}$.
We may assume that  $\mu=\tau_*\alpha$, where
$\tau \colon W\to X$ is proper and  $\alpha$ is  a product of components of Chern forms.
Consider the fibre square
\begin{equation}\label{studs3}
\begin{array}[c]{ccc}
W'  & \stackrel{\rho}{\longrightarrow} & X' \\
\downarrow \scriptstyle{\tilde\pi} & &  \downarrow \scriptstyle{p} \\
W & \stackrel{\tau}{\longrightarrow} &   X.
\end{array}
\end{equation}
Since $p$ is smooth it follows that the fibre product $W'= W\times_{X} Y$ is smooth,
cf.\ \eqref{krumbukt} below.
The pullback  $\tilde\pi^*\alpha$ is a product of Chern forms on $W'$ and thus
$\rho_*\tilde\pi^*\alpha$ is a generalized cycle on $X'$.  We claim that
\begin{equation}\label{grus}
\rho_*\tilde\pi^*\gamma=p^* \tau_*\gamma
\end{equation}
for any smooth form $\gamma$.
Taking \eqref{grus} for granted we conclude that $p^*\mu=p^*
\tau_*\alpha$ is a generalized cycle, which proves (i). 
It is enough to prove \eqref{grus} for all smooth forms $\gamma$ with
small support. Notice  that locally in $X$,
say in a small open set  $\U$,  $X'|_\U$  is biholomorphic to $\U\times \P^1$. 
Let us assume that $\tau_*\gamma$ has support in an open set
$\U\subset X$, where $X'=\U\times\P^1_t$.  Letting $\widetilde W=\tau^{-1}(\U)$,  
by the definition of fiber product,
\begin{equation}\label{krumbukt}
\widetilde W\times_\U (\U\times\P^1_t)=\{(w,x,t);\  \tau(w)=p(x,t)=x\}= \{(w, \tau(w),t);\  w\in \widetilde W\}\simeq \widetilde W\times \P^1
\end{equation}
and $\rho(w,t)=(\tau(w),t)$.   Now \eqref{grus} is obvious. 


To see (iii), note that if $\mu=\tau_*(\beta\w\alpha)$, where $\beta$ is a component of a $B$-form, then 
it follows from \eqref{grus} that
$\pi_*p^*\mu=\pi_*\rho_*(\tilde\pi^*\beta\w\tilde\pi^*\alpha)$
and hence $0$ in $\B(\Pr^{2n+1})$ since $\tilde\pi^*\beta$ is a component of a $B$-form.
\end{proof}

If $\mu_1,\mu_2\in\GZ(\Pr^n)$, then $\mu_1\times \mu_2\in\GZ(\Pr^n\times \Pr^n)$ 
by Lemma~\ref{skaldjur}, 
and by virtue of Lemma~\ref{lagom} we can make the following definition.
\begin{df}
For $\mu_1,\mu_2\in\GZ(\Pr^n)$ we define the join product $\mu_1\times_J\mu_2$ by
$$
\mu_1\times_J\mu_2:=\pi_*p^*(\mu_1\times \mu_2).
$$
\end{df}
It follows from the same lemmas that $ \mu_1\times_J\mu_2\in\GZ(\P^n)$ and, moreover, that $ \mu_1\times_J\mu_2$ is
$0$ in $\B(\P^{2n+1})$ if $\mu_1$ or $\mu_2$ is $0$ in $\B(\P^n)$. Hence,
$\mu_1\times_J\mu_2$ is well-defined for $\mu_j\in\B(\P^n)$.

\begin{ex}[Relation to the classical join]\label{unge}
Assume that $X_1,X_2\subset \P^n$ are (irreducible) analytic
sets. 
Let $\tilde p\colon \C^{n+1}\setminus\{0\}\times\C^{n+1}\setminus\{0\}\to \Pr^n\times \Pr^n$ 
and $\tilde\pi\colon \C^{2n+2}\setminus\{0\}\to\Pr^{2n+1}$ be the natural maps.
Notice that $\tilde X=\tilde p^{-1} (X_1\times X_2)$ is homogeneous in
$\C^{2n+2}$ 
and  $\tilde\pi(\tilde X)$ is the classical join of $X_1$ and $X_2$. 
%
We claim that 
\begin{equation}\label{larmar}
X_1\times_J X_2= \tilde\pi(\tilde X).
\end{equation}
Since $\mathfrak{p}\circ\tilde\pi=\tilde p$ on the common set of definition it follows that \eqref{larmar} holds outside
the union $V\subset \P^{2n+1}$ of planes where $\mathfrak p$ is not
defined. 
To prove \eqref{larmar}  
it is thus enough to show that 
$\1_V \pi_* p^* (X_1\times X_2)$ vanishes. 
In view of \eqref{linalg}, $\1_V \pi_* p^* (X_1\times X_2)=0$ if
$\1_{\pi^{-1}V} p^* (X_1\times X_2)=0$, which may be checked locally in $Bl\,\P^{2n+1}$.
We may therefore consider a subset
$\U\times \P_t^1$ of $Bl\,\P^{2n+1}$, where $\mathcal{U}\subset\P^n\times\P^n$ is open, cf.\ the proof of Lemma~\ref{lagom}. Note
that, in $\U\times \P_t^1$, $\pi^{-1}V$ is of the form $H:=\U\times
\{t_0\}$ and that $p^*(X_1\times X_2)=X_1\times X_2\times\P^1_t$. 
Thus, by the dimension principle, $\1_{\pi^{-1}V} p^* (X_1\times X_2)=\1_H (X_1\times X_2\times \P^1)=0$.
\end{ex}

\begin{ex}\label{nyckelben}
Let $\mu_1,\mu_2\in\B(\P^n)$ and assume that $\Lambda\colon\P^n\to\P^{n'}$ is a linear embedding, i.e., $\Lambda$ is induced by
an injective linear map $\tilde\Lambda\colon\C^{n+1}\to\C^{n'+1}$. Then $\tilde\Lambda\times\tilde\Lambda$ is an
injective linear map $\C^{2n+2}\to\C^{2n'+2}$ and we get a linear embedding $\mathbf{\Lambda} \colon\P^{2n+1}\to\P^{2n'+1}$.
Let $\pi'$ and $p'$ be defined in the same way as $\pi$ and $p$ in \eqref{svangd} with $n$ replaced by $n'$. 
Similarly to the proof of Lemma~\ref{lagom} one shows that $\mathbf{\Lambda}_*\pi_*p^*=\pi'_*(p')^*(\Lambda\times\Lambda)_*$
as operations on currents in $\P^n\times\P^n$. 
It follows that
\begin{equation*}
\mathbf{\Lambda}_* (\mu_1\times_J\mu_2) = \Lambda_*\mu_1 \times_J \Lambda_*\mu_2.
\end{equation*}
\end{ex}

In a similar way as above we have the mapping 
\begin{equation}\label{hunger}
\Pr^{r(n+1)-1}_{x^1,\ldots, x^r}\stackrel{\p}{\dashrightarrow} \Pr^n_{x^1}\times\cdots \times
\Pr^n_{x^r}, \quad [x^1,\ldots, x^r]\mapsto \big ([x^1], \ldots, [x^r]\big ).
\end{equation}
Let now $\pi\colon Bl\,\P^{r(n+1)-1}_{x^1,\ldots, x^r}\to \P^{r(n+1)-1}_{x^1,\ldots, x^r}$
be the blow-up of
$\P^{r(n+1)-1}_{x^1,\ldots, x^r}$ along the codimension $n$-planes  
$\{x^1=0\}, \ldots, \{x^r=0\}$ and set $p:=\mathfrak p \circ \pi$. We get a diagram analogous to
\eqref{svangd}.
As above, given $\mu_1,\ldots, \mu_r$ in $\GZ(\P^n)$ or in $\B(\P^n)$,  
we define $\mu_1\times_J\cdots\times_J\mu_r$ in
$\GZ(\Pr^{r(n+1)-1})$ or in $\B(\Pr^{r(n+1)-1})$, respectively, as $\pi_*p^*(\mu_1\times\cdots \times \mu_r)$.

\begin{prop}\label{gnet} 
If $\mu_1,\ldots, \mu_r\in\GZ(\P^n)$, then 
\begin{equation*}
\deg (\mu_1\times_J\cdots \times_J\mu_r)=\deg \mu_1\cdots\deg \mu_r.
\end{equation*}
\end{prop}

\begin{proof}
We may assume that the $\mu_j$ have pure dimension. 
There are  currents $a_j$ in $\P^n$ such that $dd^c a_j=\mu_j-(\deg\mu_j)\hat\omega^{k_j}$
if $\dim \mu_j=n-k_j$, where $\hat\omega$ is the Fubini-Study form on $\P^n$. It follows that there is a current $A$ on $\P^n_{x^1}\times\cdots\times\P^n_{x^r}$  such that
$$
dd^c A=\mu_1\times\cdots\times \mu_r-(\deg\mu_1\cdots\deg\mu_r)\hat\omega^{k_1}\times\cdots\times\hat\omega^{k_r},
$$
cf.\ Lemma \ref{skaldjur}. 
Applying $\pi_*p^*$, it is enough to show that 
$\deg (\hat\omega^{k_1}\times_J\cdots\times_J\hat\omega^{k_r})=1$;
but this is obvious if we just notice that
$\pi_*p^*$ of a hyperplane in $\P^n_{x^1}\times\cdots\times\P^n_{x^r}$
induced by a hyperplane in one of the factors $\Pr^n_{x^j}$ is a hyperplane in
$\Pr^{r(n+1)-1}_{x^1,\ldots, x^r}$
and replace each $\hat\omega^{k_j}$ by the intersection of $k_j$ generic
hyperplanes.
\end{proof}

For the last argument one can also observe that
$\log\big ((|x^1|^2+\cdots + |x^r|^2)/|x^j|^2\big )$ is a well-defined locally integrable function on $\Pr^{r(n+1)-1}_{x^1,\cdots,x^r}$ 
and that
$$
dd^c\log\big ((|x^1|^2+\cdots + |x^r|^2)/|x^j|^2\big )=\omega_{x^1,\cdots,x^r}-\pi_*p^*\omega_{x^j}.
$$

\smallskip

Let 
\begin{equation}\label{jdia}
j\colon \P^n \hookrightarrow \P^{r(n+1)-1}, \quad [x]\mapsto [x,\ldots,x].
\end{equation}
be the parametrization of the join diagonal $\Delta_J$ in
$\P^{r(n+1)-1}$ and let $\J_J$ be the associated sheaf. 
Notice that $\J_J$ is generated by  the $(r-1)(n+1)$ linear forms, i.e., sections of $L=\Ok(1)$,  
\begin{equation}\label{lucia}
\eta=(x_0^2-x_0^1,x_0^3-x_0^2,\ldots,x_0^r-x_0^{r-1}, \ldots, x_n^2-x_n^1,x_n^3-x_n^2,\ldots,x_n^r-x_n^{r-1}).
\end{equation}
Since 
\begin{equation*}
\codim\Delta_J=r(n+1)-1-n=(r-1)(n+1)
\end{equation*}
we see that $\eta$ is a minimal
generating set.

\begin{df}\label{bulletdef}
Given $\mu_1,\ldots, \mu_r\in\B(\P^n)$, 
$\mu_1\bullet \cdots \bullet \mu_r$ is the unique class in $\B(\P^n)$ such that
\begin{equation}\label{rodpenna}
j_*(\mu_1\bullet\cdots \bullet  \mu_r)= V(\J_J, L, \mu_1\times_J
\cdots \times_J \mu_r). 
\end{equation}
\end{df}

Since \eqref{hopper} is injective, $\mu_1\bullet \cdots \bullet \mu_r$ is
well-defined. It is clear that $\mu_1\bullet\cdots \bullet  \mu_r$ is commutative, multilinear, and that its
Zariski support is contained in
$|\mu_1|\cap\cdots\cap |\mu_r|$. 

If $\mu_1,\ldots, \mu_r\in \GZ(\P^n)$ denote representatives of
the corresponding classes in $\B(\P^n)$, then the right hand side of
\eqref{rodpenna} is represented by 
\begin{equation}\label{blapenna}
M^{L,\eta}\w (\mu_1\times_J \cdots \times_J \mu_r) 
\end{equation}
for any choice of $\eta$ generating $\J_J$. 
If the $\mu_j$ have pure dimensions, then
 \begin{equation}\label{apa}
d:=\dim (\mu_1\times_J\cdots \times_J \mu_r)=\sum_1^r\dim\mu_j +r -1,
\end{equation}
and thus 
$j_*(\mu_1\bullet\cdots \bullet  \mu_r)_\ell$ is represented by 
$M_{d-\ell}^{L,\eta}\w (\mu_1\times_J \cdots \times_J \mu_r).$ 

The $\bullet$-product is 
invariant in the following sense.

\begin{prop}\label{pucko2} 
Assume that $\mu_j\in\B(\Pr^n)$ and let $\Lambda\colon \Pr^n\hookrightarrow\Pr^{n'}$ be a linear embedding.
Then
$\Lambda_*(\mu_1\bullet\cdots\bullet\mu_r)=\Lambda_*\mu_1\bullet\cdots\bullet \Lambda_*\mu_r$.
\end{prop}

In particular, if $T$ is a linear automorphism of $\Pr^n$, then
\begin{equation}\label{debatt}
T_*(\mu_1\bullet\cdots \bullet \mu_r)=
T_*\mu_1\bullet \cdots \bullet T_*\mu_r.
\end{equation}

\begin{proof}
As in Example~\ref{nyckelben}, $\Lambda$ induces a linear embedding $\mathbf{\Lambda}\colon\P^{r(n+1)-1}\to\P^{r(n'+1)-1}$
and $\mathbf{\Lambda}\circ j=j\circ \Lambda$, where $j$ denotes the join diagonal in both $\P^{r(n+1)-1}$ and 
$\P^{r(n'+1)-1}$. Therefore, since $j_*$ is injective, to show the proposition it is enough to check that 
$\mathbf{\Lambda}_*j_*(\mu_1\bullet\cdots \bullet \mu_r) = j_*(\Lambda_*\mu_1\bullet \cdots \bullet \Lambda_*\mu_r)$, i.e., that
\begin{equation}\label{molnbank}
\mathbf{\Lambda}_*\big(V(\J_J, L, \mu_1\times_J\cdots \times_J \mu_r)\big) =V(\J_J, L, \Lambda_*\mu_1\times_J\cdots \times_J \Lambda_*\mu_r).
\end{equation} 

In the special case that $\Lambda$ is a linear automorphism of $\P^n$, \eqref{molnbank} follows by noticing that $\mathbf{\Lambda}^*$
in this case maps sections of $L$ to sections of $L$, preserves
$\J_J$, and, in view of a simple extension of Example~\ref{nyckelben}, that
$\Lambda_*\mu_1\times_J\cdots \times_J\Lambda_*\mu_r = \mathbf{\Lambda}_*(\mu_1\times_J\cdots\times_J\mu_r)$.

For the general case we may now  assume that $\Lambda\colon\P^n_x\to\P^{n'}_{x,y}$ 
is the map $[x]\mapsto [x:0]$.
Then $\mathbf{\Lambda}[x^1:\cdots:x^r]=[x^1:0:\cdots:x^r:0]$. Let $\eta$ be as in \eqref{lucia} and let $\eta'$ be the tuple of 
$(r-1)(n'-n)$ linear forms
$(y^{j+1}_k-y^j_k)$, $j=1,\ldots,r-1$, $k=1,\ldots,n'-n$. Then $\eta$ and $(\eta,\eta')$ define $\J_J$ in $\P^{r(n+1)-1}$ and $\P^{r(n'+1)-1}$,
respectively. By Proposition~\ref{pucko} we get
\begin{equation*}
M^{L,(\eta,\eta')}\wedge\mathbf{\Lambda}_*(\mu_1\times_J\cdots\times_J\mu_r)=
\mathbf{\Lambda}_*(M^{L,\eta}\wedge \mu_1\times_J\cdots\times_J\mu_r),
\end{equation*}
which implies  \eqref{molnbank} in view of Example~\ref{nyckelben}.
\end{proof}

%
%

\begin{prop}\label{kokosboll}
If $\mu_1, \ldots,\mu_r\in\B(\P^n)$ have pure dimensions,
then
\begin{equation}\label{bez}
\deg(\mu_1\bullet\cdots\bullet \mu_r)=\prod_1^r\deg \mu_j
-\int_{\P^{r(n+1)-1}\setminus\Delta_J}(dd^c\log|\eta|_\s^2)^{d}\w
(\mu_1\times_J\cdots \times_J \mu_r), 
\end{equation}
where $d$ is given by \eqref{apa}. 
\end{prop}

\begin{proof}
First notice that 
$j^*\hat\omega_{\P^{r(n+1)-1}}=\hat\omega_{\P^n}$, where $j$ is defined in \eqref{jdia} and $\hat\omega_{\P^M}$
denote the Fubini-Study form on $\P^M$. Therefore, for $\mu\in
\GZ_k(\P^n)$, 
\[
\deg j_*\mu = \int_{\P^{r(n+1)-1}} \hat\omega_{\P^{r(n+1)-1}}^k\wedge
j_* \mu =
\int_{\P^n} \hat\omega^k_{\P^n}\wedge
\mu =
\deg \mu. 
\]
In particular, 
$\deg (\mu_1\bullet\cdots\bullet \mu_r)= \deg \big (M^{L,\eta}\w
(\mu_1\times_J \cdots \times_J \mu_r)\big )$. 
Now, by Proposition~\ref{skottradd}, 
\[
\deg (\mu_1\bullet\cdots\bullet \mu_r)=\deg
(\mu_1\times_J\cdots\times_J\mu_r)-\deg\big (\1_{\P^{r(n+1)-1}\setminus\Delta_J}(dd^c\log|\eta|_\s^2)^{d}\w
(\mu_1\times_J\cdots \times_J \mu_r)\big ),
\]
and thus \eqref{bez} follows in view of Proposition \ref{gnet}. 
\end{proof}

The B{\'e}zout formula \eqref{likhet}
holds if and only if the last term in \eqref{bez} vanishes. This
happens if  $(r+1)(n-1)\le d$ which is the same as \eqref{vill},
cf.~the remark after Proposition~\ref{skottradd},
\eqref{lucia}, and \eqref{apa}.

However, as mentioned in the introduction, the condition \eqref{vill} is not
necessary for \eqref{likhet} to hold.
For instance, by Proposition~\ref{pucko2},
the $\bullet$-product is not
affected if we perform the multiplication in a larger $\P^{n'}$.
Thus, as mentioned already in the introduction,
the self-intersection of a $k$-plane is the $k$-plane itself,
in particular, the self-intersection of a point is the point
itself.
On the other hand, clearly the product of two
distinct points vanishes.  In this case the last term in \eqref{bez} carries the "missing mass"
in the B{\'e}zout formula.
\smallskip

We are now in position to prove Theorem~\ref{mainthm}.

\begin{proof}[Proof of Theorem~\ref{mainthm}]
The first statements, about multlilinearity, commutativity and the support,
are already discussed after Definition~\ref{bulletdef}.

Since local intersections numbers (multiplicities) are locally defined we can work in an
affinization and use the results from \cite[Sections~9 and 10]{aswy} to
prove \eqref{stenstod}.  However, we omit the details since it is also a direct consequence
of the global Proposition~\ref{baka2} below, cf.\ \eqref{blomma}
and \eqref{hennes}.
%

In the discussion after the proof of Proposition~\ref{kokosboll} is noticed that
\eqref{likhet} holds  if \eqref{vill} is fulfilled.
If $\mu_j$ are effective, then so is $\mu_1\times_J \cdots \times_J
\mu_r$, and it follows that \eqref{blapenna}, and hence
$\mu_1\bullet\cdots\bullet\mu_r$, are effective, cf.\
\eqref{epsilon2}.
Moreover $\deg \mu_j$ are positive
and the last term in \eqref{bez} is non-positive
so we get \eqref{olikhet}. 


If $\mu_1,\ldots, \mu_r$ are cycles that intersect properly, by the dimension
principle only the component of $\mu_1\bullet\cdots\bullet\mu_r$ of dimension $\rho$ is nonzero, where $\rho$ is as in
\eqref{vill}, and this is a cycle. In this case
the local intersections numbers $\epsilon_\ell(\mu_1, \ldots,  \mu_r,x)$ coincide
with the multiplicites of the proper intersection cycle
$\mu_1\cdot_{\P^n}\cdots\cdot_{\P^n}\mu_r$, cf.\ \cite[Example~10.2]{aswy}, and thus \eqref{studsmatta}
follows.
\end{proof}

We will now look at more explicit representations of the $\bullet$-product.
Recall that we have a natural Hermitian metric on $\Ok(1)$, cf.\
Section~\ref{vogelsec}, and thus, cf.~\eqref{lucia},
$$
|\eta|^2=\sum_{i=1}^{r-1} \sum_{k=0}^n|x_k^{i+1}-x_k^i|^2.
$$
From \eqref{lambdasond} 
we see that if
$\mu_i$ have pure dimension, then 
 $j_*(\mu_1\bullet\cdots\bullet \mu_r)_\ell$ is given by the value at $\lambda=0$ of
\begin{equation*}
M_k^{L,\eta,\lambda}\w (\mu_1\times_J \cdots \times_J\mu_r)
:=\dbar |\eta|^{2\lambda}\w
\frac{\partial|\eta|^2}{2\pi i|\eta|^2}\w
\big(dd^c\log|\eta|^2_\s\big)^{k-1}\w (\mu_1\times_J \cdots \times_J\mu_r),
\end{equation*}
where
$
k=d-\ell=\dim(\mu_1\times_J \cdots \times_J \mu_r)-\ell.  
$
Notice that 
\[
k\ge \dim(\mu_1\times_J \cdots \times_J \mu_r)-(\dim\mu_1+\cdots
+\dim\mu_r)= r-1\ge 1
\] 
so that term corresponding to $k=0$ in  \eqref{lambdasond} is
irrelevant here; indeed 
$\dim (\mu_1\bullet\cdots\bullet\mu_r) \leq \sum \dim\mu_j$ and so 
$\ell\leq \sum \dim\mu_j$. 

In an affinization we can also obtain the $\bullet$-product, cf.~\eqref{epsilon2},  as a limit of smooth forms
times $\mu_1\times_J\cdots \times_J\mu_r$
by the formula   
\begin{equation*}
M_k^{L,\eta}\w (\mu_1\times_J \cdots \times_J\mu_r)= 
\lim_{\epsilon\to 0}\frac{\epsilon(dd^c|\eta|_\s^2)^k}{(\epsilon+|\eta|_\s^2)^{k+1}}\w
(\mu_1\times_J \cdots \times_J \mu_r).
\end{equation*} 
When computing $M_k^{L,\eta}\w (\mu_1\times_J \cdots
\times_J\mu_r)$ 
it can be convenient to compute the SV-cycle $v^{a\cdot \eta}\wedge
(\mu_1\times_J \cdots \times_J\mu_r)$ for generic hyperplanes
$a_0\cdot\eta, a_1\cdot \eta, \ldots, a_n\cdot\eta$, $a_j\in\P^n$ and
then form the mean value,
cf.\ 
Section \ref{vogelsec}. See Section~\ref{exsektion}  for examples.

\begin{remark}\label{r2}
Assume that $r=2$. Given the standard coordinates on $\C^{n+1}$ there is a canonical choice of $\eta$ defining
$\Delta_J$, namely $\eta_j=y_j-x_j$, $j=0,\ldots,n$, cf.\ \eqref{lucia}. Thus,
given representatives of $\mu_j$, there are
canonical representatives \eqref{blapenna} of $V(\J_J, L,
\mu_1\times_J \mu_2)$, and since \eqref{linser} is injective we can
define the $\bullet$-product on the level of generalized
cycles. Indeed, given 
$\mu_1,\mu_2\in \GZ(\P^n)$, we define
$\mu_1\bullet\mu_2$ as the unique current in $\GZ(\P^n)$ such that 
\[
j_*(\mu_1\bullet \mu_2)= M^{L, \eta}\wedge (\mu_1
\times_J \mu_2).
\]

Let $T$ be a linear automorphism of $\P^{n}$ induced by a unitary mapping $\tilde T$ on $\C^{n+1}$,
let $\tilde{\mathbf{T}}=\tilde T\times \tilde T$, and let $\mathbf{T}$ be the induced linear automorphism  of $\P^{2n+1}$;
cf.\ Example~\ref{nyckelben} and the proof of Proposition~\ref{pucko2}.
Then, considering $\eta$ as a tuple of linear forms on $\C^{2n+2}$, 
$
|\tilde{\mathbf{T}}^*\eta|^2_{\C^{2n+2}}=|\eta|^2_{\C^{2n+2}}$. 
Moreover,
$dd^c\log |\eta|^2_{\C^{2n+2}} = dd^c\log |\eta|^2_\circ$, where we on the right-hand side 
consider $\eta$ as a tuple of sections of $L\to\P^{2n+1}$. Hence, 
$dd^c\log |\mathbf{T}^*\eta|^2_\circ= dd^c\log |\eta|^2_\circ$, and so 
\[
M^{L, \mathbf{T}^*\eta}\wedge (\mu_1
\times_J \mu_2)
=
M^{L, \eta}\wedge (\mu_1
\times_J \mu_2).
\]
It follows that $T_*\mu_1\bullet T_*\mu_2=T_*(\mu_1\b\mu_2)$ as generalized cycles. 
\end{remark}

\begin{remark}\label{robinson}
Consider \eqref{hunger} and the corresponding diagram \eqref{svangd}. By abuse of notation, let $\Delta_J$ denote 
the preimage under $\pi$ of the join diagonal, let $\J_J$
denote the sheaf in $Bl\,\P^{r(n+1)-1}$ corresponding to $\Delta_J$, and 
let $j$ denote the embedding of $\P^n$ in $Bl\,\P^{r(n+1)-1}$ as $\Delta_J$ induced by
\eqref{jdia}.
Since \eqref{blapenna} has support on $\Delta_J$ and $Bl\,\P^{r(n+1)-1}$ and
$\Pr^{r(n+1)-1}$ coincide in a \nbh of $\Delta_J$ we can alternatively
think of \eqref{blapenna} as a generalized cycle on $Y$. 
 


\end{remark}


%

\section{Relation to the $\cdot_{\B(\Pr^n)}$ product}\label{rela}

In this section we prove Theorem~\ref{main2}.    
For simplicity let us restrict from now on to the case $r=2$; the
general case is handled in a similar way.

Consider the mapping   
\begin{equation}\label{grymmaste} 
i^{!!}\colon\B(\P^n\times \P^n)\to \B(\P^n),\quad i_*i^{!!}\mu = c(N_{\J_\Delta}(\P^n\times \P^n))\w S(\J_{\Delta},\mu), 
\end{equation}
where  $i$ is given by \eqref{valkyria}. 
Notice that $\mu_1\cdot_{\B(\Pr^n)}\mu_2=i^!(\mu_1\times\mu_2)$ is the component of dimension $\rho$
of $i^{!!}(\mu_1\times \mu_2)$, where $\rho$ is given
by \eqref{vill}, i.e., $\rho=\dim\mu_1+\dim\mu_2-n$. 

Next, consider the mapping  
\begin{equation}\label{grymmare} 
j^{\flat}\colon\B(\P^n\times \P^n)\to \B(\P^n), \quad j_*j^{\flat}\mu = c(N_{\J_J} \P^{2n+1})\w S(\J_{J},\pi_*p^*\mu), 
\end{equation}
where we are using the notation from Section
\ref{bulletsection} and where  $j$ is given by
\eqref{jdia}.

\begin{prop}\label{baka2}  
The mappings $i^{!!}$ and $j^\flat$ coincide. 
\end{prop}

Let $\eta$ be the section \eqref{lucia} of $L=\Ok(1)$ equipped with
the Fubini-Study metric, and let
$\hat\omega$ be the first Chern
form. Then $\hat c(N_{\J_J} \P^{2n+1})=(1+\hat\omega)^{n+1}$ and thus, by
\eqref{gastel22}, 
$j^\flat$ is represented by 
\[
(1+\hat \omega)^{n+1}\wedge M^\eta\wedge \pi_*p^*\mu
=
\sum_{j\geq 0} (1+\hat\omega)^{n+1-j}\wedge M^{L,\eta}_j \wedge
\pi_*p^*\mu.
\]

Now assume that $\mu=\mu_1\times \mu_2$ and let 
%
%
$d=\dim \pi_* p^* \mu
= \dim\mu_1+\dim\mu_2 +1$, cf.\ \eqref{apa}. Note that
$\rho=d-(n+1)$. It follows that 
\begin{multline*}
\big ( \sum_{j\geq 0} (1+\hat\omega)^{n+1-j}\wedge M^{L,\eta}_j \wedge
\pi_*p^*\mu\big )_\rho 
=
\big ( \sum_{\ell\geq 0} (1+\hat\omega)^{\ell-\rho}\wedge M^{L,\eta}_{d-\ell} \wedge
\pi_*p^*\mu\big )_\rho 
=\\
\sum_{\ell\geq 0} \hat\omega^{\ell-\rho}\wedge M^{L,\eta}_{d-\ell} \wedge
\pi_*p^*\mu.  
\end{multline*}
By Definition~\ref{bulletdef}, $j_*(\mu_1\bullet\mu_2)_\ell$ is represented by 
$M^{L,\eta}_{d-\ell}\wedge\pi_* p^* \mu$ and therefore 
\begin{equation}\label{hennes}
\big ( c(N_{\J_J} \P^{2n+1})\w S(\J_{J},\pi_*p^*\mu) \big
)_\rho 
= 
j_*\sum_{\ell\geq 0} \omega^{\ell-\rho} \wedge (\mu_1\bullet \mu_2)_\ell
\end{equation} 
and thus Theorem~\ref{main2} follows from Proposition \ref{baka2}.

\begin{remark}
There are classical mappings $\A(\P^n\times\P^n)\to\A(\P^n)$ analogous to
$i^!$ and $j^\flat$. If $\mu_1$ and $\mu_2$ are cycles and $\mu=\mu_1\times\mu_2$,
then, see \cite[Example~8.4.5]{Fult}, the analogue of Proposition~\ref{baka2} holds for the component of dimension $\rho$,
which is the component of main interest also for us. However, the argument given in \cite{Fult} cannot be transferred to
to the $\B$-setting.
\end{remark}

\begin{proof}[Proof of Proposition~\ref{baka2}]
Let $Bl\,\P^{2n+1}_{x,y}$ be as in Section \ref{bulletsection}. 
%
Since $Bl\,\P^{2n+1}_{x,y}$ coincides with $\Pr^{2n+1}$ in a \nbh of $\Delta_J$,  the restrictions
of $c(N_{\J_J}{Bl\,\P^{2n+1}_{x,y}})$ and $c(N_{\J_J}{\Pr^{2n+1}})$ to
$\Delta_J$ coincide, and moreover, 
$\pi_* p^*\mu$ and $p^*\mu$ coincide on $\Delta_J$, cf.\ Remark
\ref{robinson}. 
Therefore $j^\flat$ coincides with the mapping 
\begin{equation}\label{grym} 
\B(\P^n\times \P^n)\to \B(\P^n), \quad \mu \mapsto c(N_{\J_J} Bl\,\P^{2n+1}_{x,y})\w S(\J_{J},p^*\mu), 
\end{equation}
where we are identifying 
$\Delta_J\subset Bl\,\P^{2n+1}_{x,y}$ with $\P^n$. 
Hence it suffices to prove that $i^!$ coincides with
\eqref{grym}. 

\smallskip 

Let $M=\Pr^n$ so that $\Delta=i(M)$ and $\Delta_J=j(M)$ and let $X=\P^n\times\P^n$ and $Y=Bl\,\P^{2n+1}_{x,y}$. Then 
%
\begin{equation*}
\xymatrix{
M \ar@{^{(}->}[r]^{j} \ar[d]^{id} &  Y  \ar[d]^{p}  \\ 
M    \ar@{^{(}->}[r]^{i}  & X
 }
\end{equation*}  
commutes. 
Note that $j(M)$ is a divisor in $p^{-1}i(M)$. 
Let $E\to Y$ and $F\to X$ be Hermitian vector bundles with holomorphic sections
$\phi$ and $\fff$ that define $\J_{j(M)}$ and $\J_{i(M)}$, respectively. Fix Hermitian metrics on $N_{j(M)}Y$ and $N_{i(M)}X$ and let
$\hat c(N_{j(M)}Y)$ and $\hat c(N_{i(M)}X)$  be the associated Chern forms.  
Moreover, let $\mu\in\GZ(\Pr^n\times\Pr^n)$ denote also a fixed
representative of $\mu\in \B(\P^n\times \P^n)$.

\begin{lma}\label{lupin}
Let $\L\to p^{-1} i(M)$  be the line bundle associated with the
divisor $j(M)\subset p^{-1} i(M)$. Then
\begin{equation}\label{apan1}
N_{j(M)}Y=p^* N_{i(M)}X\oplus\L  \quad \text{on} \quad j(M),
\end{equation}
and for any Hermitian metric on $\L$, 
\begin{equation}\label{apan2}
M^\fff \w\mu\sim p_* \big(\hat c(\L)\w M^{\phi} \w p^* \mu\big) \quad \text{in} \quad \GZ(X).
\end{equation}
\end{lma}

Taking this lemma for granted we can conclude the proof of Proposition~\ref{baka2}.  
We have to prove that if $\mu_1$ and $\mu_2$ are the unique elements in $\GZ(M)$ such that 
\begin{equation*}
i_*\mu_1=\hat c(N_{i(M)}{X})\w M^\fff\w\mu 
\end{equation*}
and
\begin{equation*}
j_*\mu_2=\hat c(N_{j(M)}Y) \w M^\phi\w p^*\mu,
\end{equation*}
then $\mu_1\sim\mu_2$ in $\GZ(M)$.

In view of \eqref{apan1} and \eqref{trump}  
we have
\begin{equation*}
\hat c(N_{j(M)}Y)\w M^\phi\w p^*\mu\sim \hat c(p^*N_{i(M)}X)\w \hat c(\L)\w M^\phi\w p^*\mu
\end{equation*}
in $\GZ(Y)$. 
Therefore, cf.~\eqref{sol}, 
\begin{equation}\label{apan5}
p_*\big( \hat c(N_{j(M)}Y)\w M^\phi\w p^*\mu\big)\sim \hat c(N_{i(M)}X)\w p_*\big(\hat c(\L)\w M^\phi\w p^*\mu\big).
\end{equation}
From \eqref{apan2} and \eqref{apan5} we get 
$$
p_*\big( \hat c(N_{j(M)}Y)\w M^\phi\w p^*\mu\big)\sim \hat c(N_{i(M)}X)\w M^\fff \w\mu,
$$
which means that 
$p_*j_*\mu_2\sim i_*\mu_1$  on $X$.
Since $p_*j_*=i_*$ and \eqref{hopper} is injective,  
we conclude that $\mu_1\sim \mu_2$ on $M$.
Thus Proposition~\ref{baka2} is proved.
\end{proof}

\begin{proof}[Proof of Lemma~\ref{lupin}] 
Let us use the notation $N_\fff X$ for  $N_{i(M)}X$ etc.
We first consider \eqref{apan1}. Notice that, 
with the notation from \cite[Section~7]{aeswy1}, for any columns of minimal sets of generators $s,s'$ of $\J_{i(M)}=\J_\fff$ 
at points on $iM\subset X$ there is an invertible matrix $g$ such that $s'=gs$. A section $\xi$ of the normal bundle
$N_\fff X$ can be defined as a set of holomorphic tuples $\xi(s)$ such that $g\xi(s)=\xi(gs)$ in $i(M)$,
i.e, the restriction to $i(M)$ of such matrices are transition matrices for $N_\fff X$.  
Let $t$ and  $t'$ be holomorphic functions in a neighborhood of a point on $p^{-1}i(M)$ such that both $t|_{p^{-1}i(M)}$ and 
$t'|_{p^{-1}i(M)}$ generate the
sheaf associated with the divisor $j(M)$ in $p^{-1}i(M)$. Then $t'=ht$ for a holomorphic
function $h$, which is non-vanishing on $p^{-1}i(M)$, and $h|_{p^{-1}i(M)}$ is a transition function for $\mathcal{L}$.
Moreover, $(p^*s,t)$ and $(p^*s',t')$ are minimal sets of generators for $\J_{j(M)}=\J_\phi$.  It follows that 
for given such minimal sets of generators 
at a point on $j(M)$ we have 
\[ 
\left[ \begin{array}{c}
p^*s'  \\
t'
\end{array} \right]=
\left[ \begin{array}{cc}
p^*g & 0 \\
0 & h
\end{array} \right]
\left[ \begin{array}{c}
p^*s  \\
t
\end{array} \right].
\]
Thus the restriction to $j(M)$ of the matrices 
$$
   G=
  \left[ {\begin{array}{cc}
   p^*g & 0 \\
   0 & h \\
  \end{array} } \right]
$$
are transition matrices for $N_{j(M)}Y$; it is then clear that 
\eqref{apan1} holds.
For future use let $\eta$ be the section of $\L\to p^{-1}i(M)$ that defines $j(M)$.
\smallskip

To prove \eqref{apan2} we must return to the definition of $p^*$, so let us assume that 
$\mu=\tau_*\alpha$ and recall the
fiber square \eqref{studs3}. 
We may also assume that
$W$ is chosen so that $\tau^*\fff$ is principal and hence $\rho^*\phi$ is a regular embedding of
codimension $2$ in $W'$.  We claim that
\begin{equation}\label{moral1}
 N_{\rho^*\phi}{W'} = \tilde\pi^*N_{\tau^*\fff}{W}\oplus \rho^*\L \quad \text{on} \quad \{\rho^*\phi=0\}.
 \end{equation}
In fact,   notice that $\tilde\pi^*\tau^*\fff$ combined with the section $\rho^*\eta$  
generate the same sheaf as $\rho^*\phi$.   
Arguing precisely as above for \eqref{apan1}  we then get \eqref{moral1}.

\smallskip
We now claim that
\begin{equation}\label{moral3}
[Z_{\tau^*\fff}]=\tilde\pi_*[Z_{\rho^*\phi}], 
\end{equation}
where $Z_{\tau^*\sigma}$ is the fundamental cycle of  the ideal
sheaf generated by $\tau^*\sigma$ etc. 
Since it is an equality of currents it is 
a local statement. By the dimension principle  
it is then enough to check it in an open set $\U\subset W$ where 
$Z_{\tau^*\fff}$ is smooth and $\tilde\pi^{-1}\U\simeq \U\times \Pr^1_t$ in suitable coordinates $(x,t)$
so that $\tilde\pi$ is $(x,t)\mapsto x$, cf.\ the proof of Lemma ~\ref{lagom}. Thus, we may assume
that the ideal generated by $\tau^*\fff$ is generated by $x_1^\ell$ in $\U$.  Then $\rho^*\phi$ is generated by $(x_1^\ell,t)$ and
\eqref{moral3} is reduced to the equality $\ell[x_1=0] =\tilde\pi_*(\ell[x_1=0]\times [t=0])$.

\smallskip
Next we claim that 
\begin{equation}\label{moral2} 
M^{\tau^*\fff}\sim\tilde\pi_*\big( \hat c(\rho^*\L)\w M^{\rho^*\phi}\big)
 \end{equation}
on $W$. 
In fact,  from \cite[Proposition~1.5]{aeswy1} we have
$$
M^{\rho^*\phi}=\hat s(N_{\rho^*\phi}{W'})\w[Z_{\rho^*\phi}].  
$$
By \eqref{moral1}, noting that \eqref{trump} holds for Segre forms as
well in view of \eqref{smal}, 
 we have that 
\[
\hat c(\rho^*\L)\w M^{\rho^*\phi}
\sim 
\hat c(\rho^*\L)\w \hat s(\tilde\pi^*N_{\tau^*\fff}{W})\w \hat s(\rho^*\L)\w
[Z_{\rho^*\phi}]
=
\hat s(\tilde\pi^*N_{\tau^*\fff}{W})\w [Z_{\rho^*\phi}].
\]
By \eqref{ingen} and \eqref{sol} for Segre forms, thus 
$$
\tilde\pi_*\big( \hat c(\rho^*\L)\w M^{\rho^*\phi}\big)
\sim 
\hat s(N_{\tau^*\fff}{W})\w \tilde\pi_*[Z_{\rho^*\phi}].
$$
In view of \eqref{moral3} and \cite[Proposition~1.5]{aeswy1}, now \eqref{moral2} follows.

\smallskip 
 
We can now conclude  \eqref{apan2}.
Since $\alpha$ is smooth, from \eqref{moral2} we have, cf.~\eqref{ingen}, that
$$ 
M^{\tau^*\fff}\w\alpha\sim\tilde\pi_*\big( \hat c(\rho^*\L)\w M^{\rho^*\phi}\w \tilde\pi^*\alpha\big).
$$ 
and hence, by \eqref{ponny} and the commutivity of \eqref{studs3},
$$
M^\fff\w\mu=\tau_* \big(M^{\tau^*\fff}\w\alpha\big) \sim
\tau_*\tilde\pi_*\big( \hat c(\rho^*\L)\w M^{\rho^*\phi}\w\tilde\pi^*\alpha\big) = 
p_*\rho_*\big( \hat c(\rho^*\L)\w M^{\rho^*\phi}\w\tilde\pi^*\alpha\big).
$$ 
Now, by \eqref{sol} and \eqref{ponny},  
$$
\rho_*\big( \hat c(\rho^*\L)\w M^{\rho^*\phi}\w\tilde\pi^*\alpha\big)= \hat c(\L)\w M^{\phi}\w\rho_*\tilde\pi^*\alpha,
$$
so
$$
M^{\fff}\w\mu \sim p_*\big(\hat c(\L)\w M^{\phi}\w\rho_*\tilde\pi^*\alpha\big),
$$
and since $\rho_*\tilde\pi^*\alpha=p^*\mu$, cf.\ \eqref{grus}, therefore \eqref{apan2} follows.
\end{proof}

\section{Examples}\label{exsektion}

We shall now present some further results on our products and
various examples.
We first consider an embedding  $i\colon \Pr^M\to \P^{M+1}$ as a linear hyperplane defined by the
linear form $\xi$. Let $a\in\Pr^{M+1}$ be a point outside this hyperplane and let
 $\p\colon \Pr^{M+1}\dashrightarrow\Pr^M$ be the induced projection.  If $Y$ is the blowup
 of $ \P^{M+1}$ at $a$ we have the diagram
\begin{equation*}
\xymatrix{
Y \ar[d]_{\pi} \ar[dr]^p & \\
\mathbb{P}^{M+1} \ar@{-->}[r]_{\frak{p}} & \mathbb{P}^M.
}
\end{equation*}
As in Section~\ref{bulletsection} we see that given $\mu\in\GZ(\Pr^M)$ the current
$\p^*\mu$ has a well-defined extension to an element $\pi_*p^*\mu$ in
$\GZ(\Pr^{M+1})$, cf.\ Lemma ~\ref{lagom}. 

\begin{prop}\label{grisen} 
Let $\eta$ be a tuple of linear forms on $\Pr^M$. With the notation above we have
\begin{equation}\label{gris1}
i_*\big( M^{L,\eta}\w\mu\big)=M^{L, \p^*\eta}\w i_*\mu.
\end{equation}
and
\begin{equation}\label{gris2}
i_*\big(M^{L,\eta}\w\mu\big)=M^{L, (\p^*\eta,\xi)}\w \pi_*p^*\mu.
\end{equation}
\end{prop}

\begin{proof}

Since the support of $i_*\mu$ is contained in the
hyperplane $i(\Pr^M)$ and $Y$ and $\Pr^{M+1}$ coincide in a
neighborhood of $i(\Pr^M)$, 
the right-hand side of \eqref{gris1} is well-defined. 
Now \eqref{gris1} follows from \eqref{ponny} and \eqref{ecuador} since
$\p\circ i=id$ so that $i^*\p^*\eta=\eta$.  


\smallskip
For the second equality first notice that both sides 
of \eqref{gris2} have support on $i(\P^M)$ and that $Y$ and $\P^{M+1}$ coincide in a neighborhood of 
$i(\P^M)$. For the rest of this proof let $i$ denote also the inclusion of $\P^M$ in $Y$.
Since $\eta$ defines a
regular embedding, it follows from \cite[Example~7.8]{aeswy1} that 
\begin{equation*}
i_*\big(M^{\eta}\w\mu\big)=\hat c(\pi^*L)\w M^{(p^*\eta,\pi^*\xi)}\w p^*\mu
\end{equation*}
if $\mu$ is a smooth form; here we use the standard metric on $L$.
It follows in general, by assuming that $\mu=\tau_*\alpha$, $\tau\colon W\to \Pr^M$,
and pulling back to $W$ and $W'$ according to the fibre square
\begin{equation*}
\begin{array}[c]{ccc}
W'  & \stackrel{\tau'}{\longrightarrow} & Y \\
\downarrow \scriptstyle{p'} & &  \downarrow \scriptstyle{p} \\
W & \stackrel{\tau}{\longrightarrow} &   \Pr^M,
\end{array}
\end{equation*}
cf.\ the proofs of Lemmas~\ref{lagom} and ~\ref{lupin} above. 
Since $\hat c(\pi^*L)=1+\pi^*\hat\omega$ we get
$$
i_*\big(M_j^{\eta}\w\mu\big)=M_{j+1}^{p^*\eta,\pi^*\xi}\w
p^*\mu+\pi^*\hat\omega\w M_{j}^{p^*\eta,\pi^*\xi}\w p^*\mu. 
$$
Thus, in view of \eqref{gastel12},
\begin{multline*}
i_*(M^{L,\eta}\wedge\mu) = \sum_{j\geq 0} \left(\frac{1}{1-\pi^*\hat\omega}\right)^j\wedge i_*(M_j^{\eta}\wedge\mu) \\
=
(1-\pi^*\hat\omega)\w\sum_{j\geq 1} \left(\frac{1}{1-\pi^*\hat\omega}\right)^j\wedge M_{j}^{p^*\eta,\pi^*\xi}\w p^*\mu 
+\pi^*\hat\omega \w\sum_{j\geq 0} \left(\frac{1}{1-\pi^*\hat\omega}\right)^j\wedge M_{j}^{p^*\eta,\pi^*\xi}\w p^*\mu \\
=
M^{\pi^*L,(p^*\eta,\pi^*\xi)}\wedge p^*\mu,
\end{multline*}
where we for the last equality have used that $M_{0}^{p^*\eta,\pi^*\xi}\w p^*\mu=0$ so that we may let the sum start from $j=0$;
indeed, $M_{0}^{p^*\eta,\pi^*\xi}\w p^*\mu=0$ since $\xi$ is generically non-vanishing on the Zariski support of $p^*\mu$.
Thus, \eqref{gris2} follows by applying $\pi_*$.
\end{proof}

We will now deduce a formula for $A\b\mu$ when $A$ is a linear subspace.

\begin{prop}\label{gulpenna} 
Assume that $A$ is a linear subspace of $\Pr^n$ of dimension $m$, defined by 
$n-m$ linear forms $\fff_1,\ldots,\fff_{n-m}$.  If $\mu\in\GZ_d(\Pr^n)$, then 
\begin{equation}\label{gris3}
(A\b\mu)_{d-k}=M_k^{L,\sigma}\w\mu
\end{equation}
in $\B(\Pr^n)$.
\end{prop}


\begin{proof}
Let us use the notation from Section \ref{bulletsection}. 
By \eqref{debatt} the $\bullet$-product is not affected by a 
linear change of coordinates on $\C^{n+1}_x$ and therefore we can
assume that 
$x=(x',x'')$ and $\fff=x''$. Then we need to prove that 
\begin{equation}\label{sjunga}
\mu\b[x''=0]=M^{L,x''}\w\mu
\end{equation}
in $\B(\P^n)$. Recall that $\eta=x-y$. By definition we have, cf.\ \eqref{ponny} and \eqref{ecuador},  
\begin{equation}\label{vitsippa}
j_*\big (\mu\b[x''=0]\big ) 
= 
%
M^{L,\eta}\wedge \pi_* p^* \big(\mu\times [y''=0]\big)
=
M^{L,(x'-y',x'')}\wedge \pi_* p^* \big(\mu\times [y''=0]\big). 
\end{equation} 

Recall the diagram \eqref{svangd} associated with the mapping \eqref{boljig} 
and, as in the proof of Proposition~\ref{baka2}, let $Y=Bl\,\P^{2n+1}_{x,y}$.
Consider the mapping $\p'\colon\P^{n+m+1}_{x,y'}\dashrightarrow \P^n_x\times\P^m_{y'}$, $[x,y]\mapsto ([x],[y'])$,
and let $\pi'\colon Y'\to \P^{n+m+1}_{x,y'}$ be the blow-up of
$\P^{n+m+1}_{x,y'}$ along $\{x=0\}$ and $\{y'=0\}$. 
Similarly to \eqref{svangd} we then have 
\begin{equation*}
\xymatrix{
Y' \ar[d]_{\pi'} \ar[dr]^{p'} & \\
\mathbb{P}^{n+m+1}_{x,y'} \ar@{-->}[r]_{\frak{p}'} & \mathbb{P}^n_x\times \mathbb{P}^m_{y'}.
}
\end{equation*}
Let $\iota: \P^{n+m+1}\hookrightarrow \P^{2n+1}, ~~ [x,y'] \mapsto
[x,y',0]$.  Then $\iota$ extends to a mapping $\tilde\iota\colon Y'\to Y$.
Also, let $\iota':\P^n\times \P^m \hookrightarrow \P^n \times \P^n,  ~~ ([x],[y']) \mapsto
([x],[y',0])$. 
Consider the fibre square 
\begin{equation*}
\begin{array}[c]{ccc}
Y'  & \stackrel{\tilde\iota}{\longrightarrow} & Y \\
\downarrow \scriptstyle{p'} & &  \downarrow \scriptstyle{p} \\
\P^n\times \P^m  & \stackrel{\iota'}{\longrightarrow} &   \P^n \times \P^n, 
\end{array}
\end{equation*}
cf.\ \eqref{studs3}. Notice that $\mu\times [y''=0]=\iota'_* (\mu\times 1)$.
By the same arguments as in the proof of Lemma~\ref{lagom}, we get 
\begin{equation}\label{hurra}
p^* \big (\mu\times [y''=0]\big) = p^* \iota'_* (\mu\times 1) = \tilde\iota_* (p')^*(\mu\times 1), 
\end{equation} 
cf.\ \eqref{grus}. It is straightforward to check that $\pi\circ\tilde\iota=\iota\circ\pi'$ and so,
by applying $\pi_*$ to \eqref{hurra}, we get
\begin{equation}\label{hurra2}
\pi_*p^* \big (\mu\times [y''=0]\big) = \pi_*\tilde\iota_* (p')^*(\mu\times 1)=
\iota_*\pi'_*(p')^*(\mu\times 1).
\end{equation} 
Let $p''\colon\P^n\times\P^m\to\P^n$ be projection on the first factor and set $p''':=p''\circ p'$.
Then $\mu\times 1=(p'')^*\mu$ and $(p')^*(\mu\times 1)=(p''')^*\mu$. Thus, by \eqref{hurra2},
\begin{equation}\label{hurra3}
\pi_*p^* \big (\mu\times [y''=0]\big) = \iota_*\pi'_*(p''')^*\mu.
\end{equation}
By \eqref{vitsippa}, \eqref{hurra3}, and repeated use of \eqref{gris1} we get
\begin{equation}\label{hurra4}
j_*\big(\mu\bullet [x''=0]\big) = M^{L,(x'-y',x'')}\wedge \iota_*\pi'_*(p''')^*\mu =
\iota_*\big(M^{L,(x'-y',x'')}\wedge\pi'_*(p''')^*\mu\big).
\end{equation}

Let $j':\P^n\to \P^{n+m+1}, ~~ [x]\mapsto [x,x']$, and let $\mathfrak{q}\colon \P^{n+m+1}\dashrightarrow \P^n$,
$[x,y']\mapsto [x]$. Then we have the commutative diagram
\begin{equation*}
\xymatrix{
Y' \ar[d]_{\pi'} \ar[dr]^{p'''} & \\
\mathbb{P}^{n+m+1}_{x,y'} \ar@{-->}[r]_{\frak{q}} & \mathbb{P}^n_x. 
}
\end{equation*}
By repeated use of \eqref{gris2}, with $\xi_j=x'_j-y'_j$, $j=0,\ldots,m$, we get
\begin{equation*}
M^{L,(x'-y',x'')}\wedge\pi'_*(p''')^*\mu = j'_*\big(M^{L,x''}\wedge\mu\big)
\end{equation*}
and so, by \eqref{hurra4},
\begin{equation*}
j_*\big(\mu\bullet [x''=0]\big) = \iota_*j'_*\big(M^{L,x''}\wedge\mu\big).
\end{equation*}
Since $j_*$ is injective, to finish the proof it suffices to check that we may replace $\iota_*j'_*$ by $j_*$ in the right-hand side.
Notice that $\nu:=M^{L,x''}\wedge\mu$ is a generalized cycle with support $\{x''=0\}$ so that
$\nu=i_*\nu'$ for some $\nu'\in\GZ(\{x''=0\})$, where $i\colon \{x''=0\}\hookrightarrow \P^n$ is the inclusion.
Since $\iota\circ j'\circ i=j\circ i$ we obtain
\begin{equation*}
\iota_*j'_* \nu=\iota_*j'_*i_*\nu'=j_* i_*\nu'=j_*\nu.
\end{equation*}
%

\end{proof}

 \begin{prop}\label{pyton}
Assume that $\mu\in\B(\Pr^n)$.  Then
\begin{equation}\label{kokong1}
\1_{\Pr^n}\b \mu=\mu.
\end{equation}
If $a$ is a point, then 
\begin{equation}\label{kokong2}
a\bullet\mu=\mult_a\mu\cdot [a].
\end{equation}
 \end{prop}


\begin{proof}
From Proposition \ref{gulpenna} we have that
$\1_{\Pr^n}\b \mu=M^{L,0}\w\mu=\mu$ and so  \eqref{kokong1} follows.
To see \eqref{kokong2} let 
 $\xi$ be linear forms that define $a$. By \eqref{gris3} and
 \eqref{skrutt} we have 
$a\bullet\mu=M^{L,\xi}\w\mu=  \mult_a\mu\cdot [a] $.
 \end{proof}





Let $\eta$ be a fixed choice of a tuple of linear forms defining the join diagonal $\Delta_J$ in $\P^{r(n+1)-1}$.
Then, using the notation of Section~\ref{bulletsection}, we can define a $\bullet$-product of $\mu_1,\ldots,\mu_r\in\GZ(\P^n)$ by
\begin{equation}\label{gullapp}
j_*(\mu_1\bullet\cdots\bullet\mu_r):=M^{L,\eta}\wedge \pi_*p^*(\mu_1\times\cdots\times\mu_r),
\end{equation}
cf.\ Definition~\ref{bulletdef}. With this definition, for $\mu\in\GZ(\P^n)$, \eqref{kokong1} and \eqref{kokong2} hold 
in $\GZ(\P^n)$.

\begin{prop}  \label{sumo}
Let $\eta$ be a fixed choice as above. Assume that $\mu_0,\mu_1,\ldots,\mu_r\in\GZ(\Pr^n)$ and that 
$\mu_0=\gamma\w\mu_1$
in an open set $\U\subset \P^n$, and $\gamma$ is a smooth and closed form.
Then
\begin{equation}\label{snabel}
\mu_0\bullet \mu_2\b\cdots\b\mu_r=\gamma\w (\mu_1\bullet\cdots\b\mu_r)
\end{equation}
 in $\U$.
\end{prop}

Combined with \eqref{kokong1} 
we see that 
\begin{equation}\label{bamse}
\gamma\b\mu=\gamma\w\mu
\end{equation}
in $\U$ if $\gamma\in\GZ(\Pr^n)$
is a smooth form there.

\begin{proof}
In view of \eqref{ponny} and \eqref{ecuador} we have
\begin{equation}\label{tuppa}
j_*(\mu_0\bullet \mu_2\b\cdots\b\mu_r) =\pi_* M^{L,\pi^*\eta} \wedge p^*
(\mu_0\times\mu_2\times\cdots\times \mu_r).
\end{equation} 
Now 
\[
\mu_0\times \mu_2\times\cdots\times\mu_r = 
(\gamma\times \1\times\cdots\times\1) \wedge (\mu_1\times\cdots\times\mu_r)
\] 
in $\U\times\P^n\times\cdots\times\P^n$. 
Since $\gamma\times\1\times\cdots\times\1$ is a smooth and closed form it follows from 
\eqref{lamm} that the right hand side of
\eqref{tuppa} equals 
\begin{equation}\label{angrybirds}
\pi_*\big(p^*(\gamma\times \1\times\cdots\times\1) \wedge 
M^{L,\pi^*\eta} \wedge p^*
(\mu_1\times\cdots\times \mu_r)\big) 
\end{equation}
in $\pi(p^{-1}(\U\times\P^n\times\cdots\times\P^n))$.
In a neighborhood of $\Delta_J=\{\eta=0\}$, $\mathfrak{p}$ is defined and so
\begin{equation*}
p^*(\gamma\times \1\times\cdots\times\1)=
\pi^*\mathfrak{p}^*(\gamma\times \1\times\cdots\times\1)
\end{equation*}
in a neighborhood of $\{\pi^*\eta=0\}$ in $p^{-1}(\U\times\P^n\times\cdots\times\P^n)$. 
Thus, \eqref{angrybirds} equals
\begin{equation}\label{jonnys}
\mathfrak{p}^*(\gamma\times \1\times\cdots\times\1)\wedge
M^{L,\eta} \wedge \pi_*p^*
(\mu_1\times\cdots\times \mu_r)
\end{equation}
on $j(\U)$.
Since $j^*\mathfrak{p}^*(\gamma \times \1\times\cdots\times\1)=\gamma$
in $\U$, by \eqref{gullapp}
we see that
\eqref{jonnys} equals 
\begin{equation*}
j_*(\gamma\wedge\mu_1\bullet\cdots\bullet\mu_r))
\end{equation*}
on $j(\U)$.
Using that $j_*$ is injective on currents we get \eqref{snabel}. 
\end{proof}

\begin{ex} Let $\hat\omega$ be the Fubini-Study metric form on $\P^n$. Then
$\hat\omega$ is a generalized cycle of degree $1$ and with multiplicity $0$ at each point.
Given any choice of $\eta$ as above,
it follows from Proposition~\ref{sumo} that $\hat\omega\b\hat\omega=\hat\omega\w\hat\omega$ and,
more generally,
$
\hat\omega\b \cdots \b\hat\omega=:\hat\omega^{k\b}=\hat\omega^k.
$
\end{ex}

\begin{ex}\label{theta}
Let $a=[1,0,\ldots,0]\in\Pr^n$ and let 
$\theta =dd^c\log(|x_1|^2+ \cdots + |x_n|^2)$ in $\P^n_{x_0,\ldots,
  x_n}$.  
For each $k$, 
$\theta^k$ is a well-defined positive closed current, see, e.g., 
\cite[Chapter~III]{Dem2}. 
It is an irreducible generalized cycle of dimension $n-k$ and degree $1$, with
$\mult_a\theta^k=1$ and $\mult_x\theta^k=0$ for $x\neq a$; for $k<n$,
$\theta^k$ has Zariski-support equal to $\Pr^n$ whereas
$\theta^n=[a]$, see \cite[Example~6.3]{aeswy1} and cf.~Example~\ref{predikat}.  
%
%
%
One can think of $\theta^{k}$ as
an $(n-k)$-plane through $a$ moving around $a$.  
We claim that 
\begin{equation}\label{sumo2}
\theta\b \cdots \b\theta=:\theta^{k\bullet}=\theta^{k} , \quad k\leq
n.  
\end{equation}
In fact, notice that both sides coincide outside $a$ in virtue of Proposition~\ref{sumo}.
Thus they can only differ on a generalized cycle with Zariski support at $a$, that is, 
$m[a]$ for some integer $m$.
Since the degree of $\theta$ is $1$, also the degree of $\theta^{k\bullet}$ must be $1$  by
the B{\'e}zout formula \eqref{likhet}; indeed note that $\rho$ in
\eqref{vill} in this case equals $n-k\geq 0$. Since the degree of the right hand side
is $1$ it follows that  $m=0$ and hence \eqref{sumo2} holds. 
\end{ex}

\begin{ex} Let $n=2$, let $a$ and $\theta$  be as in the previous example, and let $\ell$ be a line through $a$. Then
\begin{equation}\label{tapp}
\theta\bullet [\ell]=[a].
\end{equation}
In fact, in view of \eqref{bamse}, outside $a$, $\theta\bullet
[\ell]=\theta\w[\ell]$, which vanishes since the pullback of $\theta$ to $\ell$
vanishes.  By the same argument as in Example~\ref{theta},
using B{\'e}zout's formula \eqref{likhet}, we get \eqref{tapp}.   
\end{ex}


\begin{ex}\label{kaffekopp}
Let $\mu_1, \ldots, \mu_r$, $r\ge 2$,  be different lines through $a\in\Pk^n$. We claim that
$\mu_1\b\cdots\b \mu_r=[a]$.  In fact, since the set-theoretic intersection is $a$, the product must be
$m[a]$ for some integer $m$. Since the $\mu_j$ are effective it follows from \eqref{olikhet} that 
$m$ is $1$ or $0$.  By \eqref{stenstod} it is enough to
determine the local intersection number $\epsilon_0(\mu_1,\ldots,
\mu_r, a)$, and thus
we can assume that the $\mu_j$ are lines through $a=0$ in $\C^n$.  
In view of \eqref{lok} and 
\eqref{ecuador} 
this equals the multiplicity of $M^{L,\eta}_r\wedge
(\mu_1\times\cdots\times\mu_r)$, where $\eta$ is a tuple of linear
forms defining the diagonal
in $(\C^n)^r=\C^{n}\times\cdots\times \C^n$. 
This, in turn, can 
be computed by intersecting $\mu_1\times\cdots\times\mu_r$ by 
$r$ generic hyperplanes $\div (\alpha\cdot \eta)$, see Section
\ref{vogelsec}. 
Doing this, we get
$[0]$ with multiplicity $1$, which proves the
claim. 
%
\end{ex}

\begin{ex}\label{snurrsnitt}
Let $G$ be the graph in $\C^6_{x,y}=\C^6_{x_1,x_2,x_3,y_1,y_2,y_3}$ of the function 
\begin{equation*}
\C^3_x\to\C^3_y, \quad
(x_1,x_2,x_3)\mapsto (x_1x_3, x_2x_3, x_3^2),
\end{equation*} and let
$Z$ be the closure in $\Pr^6_{x_0,x,y}$.  Clearly $Z$ is irreducible of dimension $3$.
We want to  compute $A\b Z$, where $A=\{y=0\}$. 
By \eqref{gris3},
$$
(A\b Z)_{3-k}=M_k^{L,y}\w[Z].
$$
In view of Section \ref{vogelsec} we can compute the right hand side
by successively intersecting $[Z]$ by hyperplanes $\div h_j$, where
$h_1=\alpha\cdot y, h_2=\beta\cdot y$, and $h_3=\gamma\cdot y$
for generic $\alpha, \beta, \gamma\in \P^2$, and then taking
averages. 

The map $\P^3\dashrightarrow\P^6$, $[t_0,t_1,t_2,t_3]\mapsto [t_0^2,t_0t_1, t_0t_2, t_0t_3, t_3t_1, t_3t_2, t_3^2]$,
lifts to an injective holomorphic map from the blow-up $Y=Bl_{t_0=t_3=0}\P^3$ to $\P^6$ with image $Z$. Then $Z$
can be parametrized by two copies of $\P^2\times\C$,
\begin{equation*}
\P^2\times \C\ni([s,t_1,t_2],\sigma) \mapsto [s,t_1,t_2,s\sigma, \sigma t_1, \sigma t_2, \sigma^2s]\in Z
\end{equation*}
\begin{equation*}
\P^2\times \C\ni([u,t_1,t_2],v) \mapsto [uv^2,vt_1,vt_2,uv, t_1, t_2,u]\in Z,
\end{equation*}
identified by $s=uv$, $s\sigma=u$. Let $Z_1$ and $Z_2$ be the image of the first and second map, respectively.
Since $Z_2\cap A=\emptyset$, the SV-cycle we are to compute is contained in $Z_1$.

Expressed in the $([s, t_1, t_2], \sigma)$-coordinates, $A=\{\sigma t_1=\sigma t_2=\sigma^2s=0\}=\{\sigma=0\}$ and so, clearly,
$v_0^h\wedge[Z]=0$, cf.\ \eqref{badanka}. Moreover, $\div h_1$ is given by
\begin{equation*}
\sigma (\alpha_1t_1 + \alpha_2t_2 + \alpha_3s\sigma) =0.
\end{equation*}
Hence
$\div h_1$ has two irreducible components; the component $\sigma=0$ is contained in $A$ and thus contributes to $v_1^h\wedge [Z]$
whereas the component $\alpha_1t_1 + \alpha_2t_2 + \alpha_3s\sigma=0$ is not contained in $A$. Intersecting the latter component by
$\div h_2$ gives
\begin{equation*}
\alpha_1t_1 + \alpha_2t_2 + \alpha_3s\sigma = 
\sigma(\beta_1t_1 + \beta_2t_2 + \beta_3s\sigma) =0.
\end{equation*}
Again we get two irreducible components. The component $\{\sigma=\alpha_1t_1+\alpha_2t_2=0\}$ 
is contained in $A$ and contributes to 
$v_2^h\wedge [Z]$ while the component $\{\alpha_1t_1 + \alpha_2t_2 + \alpha_3s\sigma=\beta_1t_1 + \beta_2t_2 + \beta_3s\sigma=0\}$
is not contained in $A$. Intersecting the latter one by $\div h_3$ gives
\begin{equation*}
\alpha_1t_1 + \alpha_2t_2 + \alpha_3s\sigma=\beta_1t_1 + \beta_2t_2 + \beta_3s\sigma=
\sigma(\gamma_1t_1 + \gamma_2t_2 + \gamma_3s\sigma)=0.
\end{equation*}
The case $\sigma\neq 0$ forces $t_1=t_2=s=0$, which is impossible. The other case gives $2$ times the point $\{\sigma=t_1=t_2=0\}$
as contribution to $v_3^h\wedge [Z]$. 

We thus get the SV-cycle
\begin{equation}\label{hundra}
v^h\wedge [Z] = P + L_\alpha + 2a,
\end{equation}
where $P=\{x_3=y=0\}$, $L_\alpha=\{x_3=y=\alpha_1x_1+\alpha_2x_2=0\}$, and $a=[1,0,0,0,0,0,0]$ 
expressed in the original $x_0,x,y$-coordinates. 
Taking the average of \eqref{hundra} over $(\alpha, \beta, \gamma) \in (\P^2)^3$ we get 
\[
A\bullet Z = M^{L, y}\wedge [Z] = P +  \mu +2[a],
\]
where $\mu$ is the generalized cycle $[x_3=y=0]\w
dd^c\log(|x_1|^2+|x_2|^2)$ obtained as the average of $L_\alpha$. 

Note that the degree of $A\b Z$ is $4$ since each term has degree $1$
except for the double point $2[a]$. Thus, in view of \eqref{likhet}, $\deg Z=4$; indeed $\rho$ in \eqref{vill} is $0$ in
this case. 
Moreover, by \eqref{stenstod}, the local intersection numbers at $a$ are 
$\epsilon_0(A, Z, a)=\mult_a 2[a]=2$, $\epsilon_1(A, Z,a)=\mult_a\mu=1$, and $\epsilon_2(A, Z, a)=\mult_a P =1$. Here
we have used that $\mu$ has multiplicity $1$ at $a$ since it is a mean value of lines through $a$ in
the $4$-plane $\{x_3=y=0\}$, cf.~Example~\ref{theta}. 
\end{ex}


%
%

%

We now give an example that shows that the $\bullet$-product is not associative.

\begin{ex}\label{kokong} 
 Consider the hypersurface  $Z=\{x_2x_1^m-x_3^2x_0^{m-1}=0\}$ in $\P^3$,
let $H_2=\{x_2=0\}$ and $H_3=\{x_3=0\}$.     
Since $H_2$ and $Z$ intersect properly, 
$$
H_2\bullet Z=H_2\cdot_{\P^3} Z = 2\{x_2=x_3=0\}+(m-1)\{x_0=x_2=0\}
$$
cf.\ \eqref{studsmatta}. 
Let $A= \{x_2=x_3=0\} $. It follows from Proposition~\ref{pucko2} and
\eqref{kokong1} that
$\{x_3=0\}\bullet A=A;$ 
this can also be verified by a symmetry argument and the B{\'e}zout
formula \eqref{likhet}. 
Moreover, $\{x_3=0\}$ and $\{x_0=x_2=0\}$ intersect properly and the
intersection is $b=[0,1,0,0]$. Thus 
\begin{equation}\label{full1}
H_3\bullet(H_2\bullet Z)=
2 A+(m-1) [b].
\end{equation}
Next note that  $H_3\bullet H_2=A$.   It is showed in \cite[Example~11.5]{aswy}
that the 
local
intersection number for $A$ and $Z$ in dimension $0$ is $m$ at
$a=[1,0,0,0]$, 
and $1$ in dimension
$1$ at all points $x\in A$. 
It follows that $A$ and $m[a]$ are components of $A\bullet
Z$. Moreover, since $A$ and $Z$ are effective, by Theorem
\ref{mainthm}, $A\bullet Z$ is effective and of degree at most $\deg A
\cdot \deg Z = m+1$. Hence 
\begin{equation}\label{full2}
(H_3\bullet H_2)\bullet Z=A\bullet Z= A+m [a].
\end{equation}
%
\end{ex}

It follows that neither $\cdot_{\B(\Pr^N)}$  is associative in $\B(\Pr^n)$.  In fact, 
it follows from \eqref{full1}, \eqref{full2}, and Theorem~\ref{main2}, that 
$$
H_3\cdot_{\B(\Pr^n)}(H_2\cdot_ {\B(\Pr^n)} Z)=2\omega \w A+(m-1) [b],
$$
whereas
$$
(H_3\cdot_{\B(\Pr^n)} H_2)\cdot_{\B(\Pr^n)} Z= \omega\w A+m [a]
$$
and these right-hand sides are not equal in $\B(\P^n)$.


\begin{ex}\label{hatt} 
Let $\gamma$ be a smooth curve in $\Pr^2$ of degree $d$.  It is well-known, see, e.g., \cite{aswy}, that 
local intersection numbers are biholomorphic
invariants. Therefore, since the $\bullet$-self-intersection of a
line is the line itself, cf.\ the discussion after Proposition~\ref{kokosboll}, it follows from \eqref{stenstod} that at each $x\in \gamma$,
$\mult_x(\gamma\bullet\gamma)_1=1$ and
$\mult_x(\gamma\bullet\gamma)_0=0$.   
Thus, since $|\gamma\bullet\gamma|\subset \gamma$, in view of the
dimension principle, 
$\gamma\bullet\gamma=\gamma+\mu$
where $\mu$ has dimension $0$ and Zariski support equal to $\gamma$. 
%
By the B{\'e}zout formula \eqref{likhet} the degree of $\mu$ must be $d^2-d$.  
We can think of $\mu$ as $d^2-d$ points moving
around on $\gamma$.
\end{ex}

\begin{ex}\label{kraka} 
We want to compute the $\bullet$-self-intersection of a 
curve $Z$ in
$\P^2$. 
Assume that
$Z=\{F=0\}$ where $F$ is a section of $\Ok(d)$ with differential generically non-vanishing
on $Z$.  Let $\eta_j=y_j-x_j$, $j=0,1,2$,  on
$\P^5_{x,y}=\P^2_x\times_J\P^2_y$. 
Then $\eta$ defines the join diagonal $\Delta_J$. 
Following Section \ref{vogelsec} we can compute $M^{L, \eta}\wedge
(Z\times_J Z)$ by successively intersecting $Z\times_J Z$ by
hyperplanes $\div h_j$, where $h_j=\eta\cdot \alpha^j$ for generic
$\alpha^j\in \P^2$, and then
averaging over $\alpha=(\alpha^1,\alpha^2,\alpha^2)\in (\P^2)^3$. 
Note that we can write 
$$
F(y)-F(x)=\eta_0 A_0+\eta_1 A_1+\eta_2 A_2
$$
for suitable homogeneous forms $A_j$, 
and thus 
\[
Z\times_J Z=\{F(x)=0, F(y)=0\}= 
\{F(x)=0,\ \eta_0 A_0+\eta_1 A_1+\eta_2 A_2=0\}, 
\]
cf.\ Example \ref{unge}. 
It turns out that 
\[
[\div h_2]\wedge [\div h_1] \wedge (Z\times_J Z)=
\{
F(x)=0, 
\eta_2(\beta_0A_0+\beta_1 A_1+\beta_2 A_2)=0, 
\eta_1=\gamma_1\eta_2,  \eta_0=\gamma_0\eta_2\}
\]
for some $\beta, \gamma\in \P^2$. 
%
The second equation gives rise to two components. The component
corresponding to $\eta_2=0$ is contained in $\Delta_J$ and equals
\[
\{F(x)=0, \eta=0\}=\{F=0\} \cap \Delta_J = v_2^h\wedge (Z\times_J Z) =
j_* Z,
\]
where $j$ is the parametrization \eqref{jdia} of $\Delta_J$. 
Next, since 
$
A_j=F_j:=\partial F/\partial x_j
$
on $\Delta_J$ we get that 
\begin{multline}\label{flaggstang}
v_3^h\wedge (Z\times_J Z)= 
[\div h_3]\wedge \{
F(x)=0, ~ 
\sum_{j=0}^2\beta_jA_j=0, ~
\eta_1=\gamma_1\eta_2, ~ \eta_0=\gamma_0\eta_2  
\}
\\
=
\{F(x)=0,\   \sum_{j=0}^2\beta_jF_j=0,\ \eta=0\}. 
\end{multline} 
The curve defined by $\beta_0  F_0+\beta_1 F_1+\beta_2 F_2$ is a so-called {\it polar curve}
to $Z$; it is clear that it passes through all singular points
$a_1\ldots, a_r$  of $Z$,
since the gradient must vanish there. More precisely, in view of the
B{\'e}zout formula \eqref{likhet}, for 
generic $\beta$, 
\[
v_3^h\wedge (Z\times_J Z) = m_1[a_1]+\cdots+ m_r[a_r]+r_\beta, 
\]
where $m_j$ are the multiplicities of $a_j$ and $r_\beta$ are
$d^2-d-(m_1+\cdots + m_r)$ points on $Z$ depending on $\beta$, cf.\
Example \ref{hatt}. 
Thus, taking averages over $\alpha\in(\P^2)^3$, we get that 
\begin{equation}\label{hissa}
Z\bullet Z = Z + m_1[a_1]+\cdots m_r[a_r] + \mu,
\end{equation} 
where $j_*\mu$ is the average of the $r_\beta$. In particular, $\mu$ has
dimension $0$, Zariski-support equal to $Z$, and degree $d^2-d-(m_1+\cdots + m_r)$. 
Moreover, in view of Example~\ref{hatt}, $\mu$ has multiplicity $0$ at each
point.



\end{ex}

Let us now consider a simple cusp.

\begin{ex}
\label{cusp}
Let us consider the situation of the previous example 
and let $F=x_1^3-x_0x_2^2$ so that $Z\subset \P^2$ is a cusp 
with a singularity only at the point $a=[1,0,0]$.
Now 
\[
v_3^h\wedge (Z\times_J Z)= 
\{x_1^3-x_0x_2^2=0,\quad \beta_0 x_2^2+\beta_1 x_1^2+\beta_2 x_0x_2=0,
\eta=0\}
\]
for some $\beta\in \P^2$, 
see \eqref{flaggstang}.  
For generic choices of $\alpha\in (\P^2)^3$, $\beta_2\neq 0$ and we
can identify this with the set of points 
\[
\varrho_\beta = \{x_1^3-x_0x_2^2=0,\ \beta_0 x_2^2+\beta_1 x_1^2+x_0x_2=0\}\subset \P^2. 
\]
To compute the order of the zero at $a$, we can use affine coordinates
and thus let $x_0=1$. 
Then 
$
\varrho_\beta=\{x_1^3-x_2^2=0,\  \beta_0 x_2^2+\beta_1 x_1^2 + x_2=0\}. 
$
If we choose new coordinates $z_1=x_1,\ z_2=x_2+\beta_0x_2^2+\beta_1x_1^2$,
then $x_2=z_2+\Ok(z^2)$, and thus $\varrho_\beta$ is defined by the equations
$$
z_1^3-\big (z_2+\Ok(z^2)\big )^2=0,\quad z_2=0. 
$$
Hence the zero at $a=(0,0)$ has order $3$. In fact, for a complete intersection, as here, 
the order of the zero coincides with the degree of the associated mapping. 
From \eqref{hissa} we conclude that 
\begin{equation}\label{anka1}
Z\bullet Z=Z+3 [a]+\mu,
\end{equation}
where $\mu$ has dimension $0$, Zariski-support equal to $Z$,
multiplicity $0$ at each point, and degree $3$.
\end{ex}

\begin{ex}\label{cusp2}
Let $Z\subset \P^2$ be the cusp as in the previous example.
In view of Theorem~\ref{main2} and \eqref{anka1} we get
\begin{equation}\label{anka2}
Z\cdot_{\B(\Pr^2)}Z=\omega\w[Z]+3[a]+\mu.
\end{equation}
Since $Z$ is a regular embedding in $\Pr^2$ we can also form the
product $Z\diamond_{\B(\P^2)} Z$.
Let $\J\to \Pr^2$ be the sheaf defining $Z$. If $i\colon Z\hookrightarrow \Pr^2$, then  $i^*\J_Z=0$ so that
 $S(\J_Z,Z)=S(0,Z)=[Z]$, cf.\ Section \ref{byggnad}. Moreover,
 $N_Z\P^2=\Ok(3)|_Z$, so that $c_1(N_Z\P^2)=3\omega$.  Thus 
\begin{equation}\label{anka3}
Z\diamond_{\B(\Pr^2)} Z=\big(c(N_Z{\Pr^2})\w
S(\J_Z,Z)\big)_0=3\omega\w[Z], 
\end{equation}
cf.\ Definition \ref{strunta}. 
Notice that \eqref{anka2} and \eqref{anka3} do not coincide in $\B(\Pr^2)$. 
For instance,  the first one has multiplicity
$3$ at $a$, whereas the second one has multiplicity $0$ at $a$.

However, in view of Proposition~\ref{stupa} their images in $\widehat H^{2,2}(Z)$ coincide. Clearly the image of
$Z\diamond_{\B(\Pr^2)} Z$ is represented by the restriction to $Z$ of the form $3\omega$.  It is easy to see
that $3a$ is cohomologous with $\omega$ on $Z$ as
$$
3[a]-\omega\w[Z]=\mult_a Z\cdot [a] -\omega\w[Z]=dd^c(\log(|z_1|^2/|z|^2)[Z]).
$$
It is somewhat less obvious that $\mu$ is cohomologous with $\omega$ on $Z$. 
\end{ex}

\noindent Example \ref{cusp2} also shows that the self-intersection
formula, Proposition \ref{sparrow}, does not generalize to non-smooth $Z$.





\end{document}